\documentclass[a4paper,11pt]{article}

\usepackage[utf8]{inputenc}
\usepackage[british]{babel}
\usepackage{fancybox}
\usepackage[mono=false]{libertine}
\usepackage[T1]{fontenc}

\usepackage{amsthm}
\usepackage[normalem]{ulem}
\usepackage[cmintegrals,libertine]{newtxmath}
\usepackage[cal=euler, scr=boondoxo]{mathalfa}
\useosf

\usepackage{microtype}

\usepackage{numprint}

\usepackage[margin=2.5cm]{geometry}

\usepackage{multirow}

\usepackage{tikz}

\usepackage{graphicx}
	\graphicspath{{./images/}}

\usepackage{pdfpages}

\usepackage[font=sf]{caption}
\usepackage{hyperref}
\hypersetup{
	colorlinks=true,
		citecolor=blue!60!black,
		linkcolor=red!60!black,
		urlcolor=green!40!black,
		filecolor=yellow!50!black,
	breaklinks=true,
	pdfpagemode=UseNone,
	bookmarksopen=false,
}

\usepackage{enumitem}

\newtheorem{thm}{Theorem}
\newtheorem{lem}{Lemma}
\newtheorem{prop}{Proposition}

\theoremstyle{definition}
\newtheorem{rem}{Remark}

\newcommand{\ensembles}[1]{\mathbb{#1}}
	
	\newcommand{\Z}{\ensembles{Z}}
	\newcommand{\R}{\ensembles{R}}

\newcommand{\ind}[1]{\ensembles{I}_{\{#1\}}}
	\renewcommand{\P}{\ensembles{P}}
	\newcommand{\E}{\ensembles{E}}

\renewcommand{\Pr}[1]{\P\left(#1\right)}

\newcommand{\Es}[1]{\E\left[#1\right]}

\renewcommand{\le}{\leqslant}
\renewcommand{\leq}{\leqslant}
\renewcommand{\ge}{\geqslant}
\renewcommand{\geq}{\geqslant}

\newcommand{\ex}{\mathrm{e}}

\newcommand{\q}{\mathbf{q}}
\newcommand{\e}{\overline{\mathfrak{e}}}

\newcommand{\dgr}{d_{\mathrm{gr}}}

\newcommand{\hBall}{\overline{\mathrm{Ball}}}

        \newcommand{\map}{\mathfrak{m}}

\newcommand{\rootface}{f_{ \mathrm{r}}}

        \newcommand{\Map}{\mathfrak{M}}

\newcommand{\cv}[1][n]{\enskip\xrightarrow[#1 \to \infty]{}\enskip}

\newcommand{\cvproba}[1][n]{\enskip\xrightarrow[#1 \to \infty]{(\P)}\enskip}

\title{Infinite stable Boltzmann planar maps are subdiffusive}

\DeclareSymbolFont{extraup}{U}{zavm}{m}{n}
\DeclareMathSymbol{\vardspade}{\mathalpha}{extraup}{81}
\DeclareMathSymbol{\varheart}{\mathalpha}{extraup}{86}
\DeclareMathSymbol{\vardiamond}{\mathalpha}{extraup}{87}
\DeclareMathSymbol{\varclub}{\mathalpha}{extraup}{84}

\makeatletter
\renewcommand*{\@fnsymbol}[1]{\ensuremath{\ifcase#1\or  \vardspade \or \vardiamond \or \varheart\or \varclub \or
   \mathsection\or \mathparagraph\or \|\or **\or \dagger\dagger
   \or \ddagger\ddagger \else\@ctrerr\fi}}
\makeatother

\author{
	Nicolas \textsc{Curien}\thanks{Laboratoire de Math\'{e}matiques, Univ. Paris-Sud, Universit\'{e} Paris-Saclay and IUF.\hfill  \href{mailto:nicolas.curien@gmail.com}{\texttt{nicolas.curien@gmail.com}}} 
\qquad\&\qquad
	Cyril \textsc{Marzouk}\thanks{Centre de Math\'{e}matiques Appliqu\'{e}ees, \'{E}cole polytechnique.\hfill  \href{mailto:cyril.marzouk@polytechnique.edu}{\texttt{cyril.marzouk@polytechnique.edu}}}
}

\begin{document}

\maketitle

\begin{abstract}
The infinite discrete stable Boltzmann maps are generalisations of the well-known Uniform Infinite Planar Quadrangulation in the case where large degree faces are allowed. We show that the simple random walk on these random lattices is always subdiffusive with exponent less than $1/3$. Our method is based on stationarity and geometric estimates obtained via the peeling process which are of own interest.
\end{abstract}

\section{Introduction}
Since the introduction of the Uniform Infinite Planar Triangulation (UIPT) by Angel \& Schramm, the study of random infinite planar maps has been a prolific activity. Generalisations of the UIPT with random ``stable'' faces have been introduced following the work of Le Gall \& Miermont but even basic properties of these lattices remain unknown. In particular, the behaviour of the simple random walk on infinite random planar maps has attracted a lot of attention in recent years, see the beautiful lecture notes of Nachmias on the subject~\cite{Nachmias:St_Flour}. In this paper, we revisit the subdiffusivity properties of the UIPQ (also called \emph{anomalous diffusion}) and more generally bipartite critical stable Boltzmann planar maps and we show a universal upper bound on the subdiffusivity exponent of $1/3$. 

\paragraph{Infinite discrete stable maps.} As usual in the field, we will only consider rooted (i.e. equipped with a distinguished oriented edge) bipartite (i.e. whose faces have even degree) planar maps. The second condition should only be technical. Given a sequence $ \q= (q_{k})_{ k \ge 1}$ of non-negative numbers, with $q_k > 0$ for some $k \ge 2$, we consider a random Boltzmann map $\Map$ whose law is prescribed by the following formula: for any finite bipartite planar map $\map$,
\[\Pr{\Map= \map} 
= \frac{1}{W_\q} \prod_{f\in \mathrm{Faces}(\map)} q_{ \deg(f)/2},\]
where $W_\q$ is a normalising constant; obviously, we restrict our attention to \emph{admissible} weight sequences $\q$, for which $W_\q < \infty$ so the above display is well-defined. We shall consider further \emph{critical weight sequences of type $a \in (3/2;5/2]$}, see~\cite[Chapter~V]{Curien:StFlour} and Section~\ref{sec:nu} below for details. Under the criticality assumption one can define  a random infinite bipartite map  $\Map_\infty$ with one end  as the local  limit in distribution of $  \Map$ conditioned to have a large size~\cite{Bjornberg-Stefansson:Recurrence_of_bipartite_planar_maps,Stephenson:Local_convergence_of_large_critical_multi_type_Galton_Watson_trees_and_applications_to_random_maps}. Furthermore, $ \q$ is of type $a \in (3/2;5/2)$ if the degree of the face $ \rootface$ adjacent to the right of the root edge, hereafter called the \emph{root face}, in $\Map_\infty$ satisfies $\P( \deg( \rootface) \ge n) \sim C \cdot n^{-a+3/2}$ for some $C \in (0,\infty)$ and $\deg( \rootface)$ has finite mean if $a = 5/2$. In particular, when $a \in (3/2;5/2)$ the map $\Map_\infty$ possesses large faces; also the case $a = 5/2$ includes the UIPQ.

Let $\dgr$ denote the graph distance in $\Map_\infty$. Conditional on $\Map_\infty$, let us sample a simple random walk started from the origin $\rho$ of the root edge and let $X_0, X_1,\dots$ be the vertices visited by this walk. Our main result shows that for any $a \in (\frac{3}{2}, \frac{5}{2}]$ the walk is always subdiffusive with exponent at most $1/3$:

\begin{thm}\label{thm:sous_diff_primal}
Let $\q$ be a critical weight sequence of type $a \in  (\frac{3}{2}, \frac{5}{2}]$. Under the annealed law of the map together with the random walk, we have
\[\frac{\sup_{0 \le k \le n}\dgr(X_0, X_k)}{n^{1/3} \log n} \cvproba 0.\]
\end{thm}

One can also consider a walk $X_0^\dag, X_1^\dag, \dots$ on the dual map $\Map_\infty^\dag$ (which walks on the faces of $\Map_{\infty}$). We proved in~\cite[Corollary~3]{Curien-Marzouk:Markovian_explorations_of_random_planar_maps_are_roundish} that for $a \in (2, \frac{5}{2}]$, the so-called ``dilute regime'', we have
\begin{equation}\label{eq:sous_diff_dual}
\lim_{R \to \infty} \Pr{X_i^\dag \in \hBall(\Map_\infty^\dag, R) \text{ for all } i \le R^{\frac{a-1}{a-2}} \log^{-1} R} = 1,
\end{equation}
where $\hBall(\Map_\infty^\dag, R)$ is the hull of the ball of radius $R$ in $\Map^{\dagger}$ obtained by keeping all faces at dual graph distance at most $R$ from the root face as well as the finite regions they enclose (recall that $\Map_\infty$ is one-ended). Note that $\frac{a-2}{a-1} \le \frac{1}{3}$ so if one could replace the hull of balls by the balls ---~which we believe is true when $a > 2$~--- this would imply a subdiffusive behaviour for the walk on the dual map in this regime.  Our proof of Theorem~\ref{thm:sous_diff_primal} can probably be adapted to this walk and would yield a similar result to~\eqref{eq:sous_diff_dual} (still with hulls of balls) but with a smaller exponent $\frac{2a-7/2}{a-2}$.

In the case $a < 2$, the so-called ``dense regime'', the walk on the dual displaces very slowly:
let us denote by $\dgr^\dag$ the graph distance on $\Map_\infty^\dag$, then

\begin{thm}\label{thm:sous_diff_dual_dense}
Let $\q$ be a critical weight sequence of type $a \in  (\frac{3}{2}, 2)$. There exists a constant $\delta > 0$ such that under the annealed law of the map together with the random walk, we have
\[\Pr{\sup_{0 \le k \le n}\dgr^\dagger(X_0^\dagger, X^\dagger_k) \le \delta \log^2 n} \cv 1.\]
\end{thm}

\begin{rem}\label{rem:attraction_stable}
One can consider critical weight sequences such that the law of the degree of the root face belongs to the domain of attraction of an $(a-3/2)$-stable law in the broad sense, allowing corrections with slowly varying functions. 
All the results of this paper generalise verbatim, except that the polylogarithms are replaced by other slowly varying functions. Although this only requires mild modifications in the proofs, appealing to basic properties of such functions, we refrained to include it in order to lighten the exposition. 
In another direction, we wonder whether Theorem~\ref{thm:sous_diff_primal} can be extended to all \emph{critical} infinite Boltzmann maps of the plane, without stable tail behaviour for the weight sequence.
\end{rem}

\paragraph{Known results.} Subdiffusive behaviour, or anomalous diffusion for random walks has been popularised by De~Gennes~\cite{DG76} under the name ``the ant in the labyrinth''. This triggered a lot of work around random walk on (critical) percolation clusters and the Alexander-Orbach conjecture, see~\cite{Kumagai:Random_walks_on_disordered_media_and_their_scaling_limits} and the references therein. In the context of random planar maps, the first result on subdiffusivity of the simple random walk on random planar maps was obtained in~\cite{Benjamini-Curien:Simple_random_walk_on_the_uniform_infinite_planar_quadrangulation_subdiffusivity_via_pioneer_points} which considers the UIPQ (which is $\Map_\infty$ when $q_k = \frac{1}{12} \ind{k=2}$, which falls in the case $a = \frac{5}{2}$): Relying on the \emph{peeling process} to study the \emph{pioneer points} of the walk, they show an upper bound of $1/3$ on the subdiffusive exponent, i.e. that $ \sup_{k \le n}\dgr(X_{0},X_{k}) \le n^{1/3+ o(1)}$ with high probability.

Very recently, we adapted this approach to the present context of infinite stable maps~\cite[Corollary~2]{Curien-Marzouk:Markovian_explorations_of_random_planar_maps_are_roundish} and showed an upper bound of $ \frac{1}{2a-2}$ for the subdiffusive exponent in the case of  weight sequence of type $ a \in (2, \frac{5}{2}]$. This should be compared with the work of Lee~\cite[Theorem~1.9]{Lee:Conformal_growth_rates_and_spectral_geometry_on_distributional_limits_of_graphs} which considers more generally unimodular random planar graphs with $d > 3$ volume growth exponent; informally\footnote{Actually, as pointed out in~\cite{Lee:Conformal_growth_rates_and_spectral_geometry_on_distributional_limits_of_graphs}, Theorem~1.9 there does not apply directly to random maps and one uses the more involved Theorem~1.15.}, since $d = 2a-1$ here~\cite[Proposition~2]{Curien-Marzouk:Markovian_explorations_of_random_planar_maps_are_roundish}, Theorem~1.9 in~\cite{Lee:Conformal_growth_rates_and_spectral_geometry_on_distributional_limits_of_graphs} reads $\E[\dgr(X_0, X_n)] \le n^{\frac{1}{2a-2} + o(1)}$ for $a \in (2, \frac{5}{2}]$.
The bound~\eqref{eq:sous_diff_dual} on the dual map is also obtained via the pioneer points approach.

Let us mention that the bound on the subdiffusivity exponent from the pioneer points one can be slightly sharpened, see~\cite{Curien-Marzouk:How_fast_planar_maps_get_swallowed_by_a_peeling_process} for Boltzmann maps with bounded face degrees, but the argument can be generalised with more effort. Nonetheless this improvement does not yield the exact exponent; in the case of the (type II) UIPT, using a Liouville Quantum Gravity approach, Gwynne \& Miller~\cite{Gwynne-Miller:Random_walk_on_random_planar_maps_spectral_dimension_resistance_and_displacement} and then Gwynne \& Hutchcroft~\cite{Gwynne-Hutchcroft:Anomalous_diffusion_of_random_walk_on_random_planar_maps} obtained the exact (lower and upper bound respectively) exponent $1/4$ which was conjectured in~\cite{Benjamini-Curien:Simple_random_walk_on_the_uniform_infinite_planar_quadrangulation_subdiffusivity_via_pioneer_points}.

\begin{figure}[!ht] \centering
\includegraphics[width=.8\linewidth]{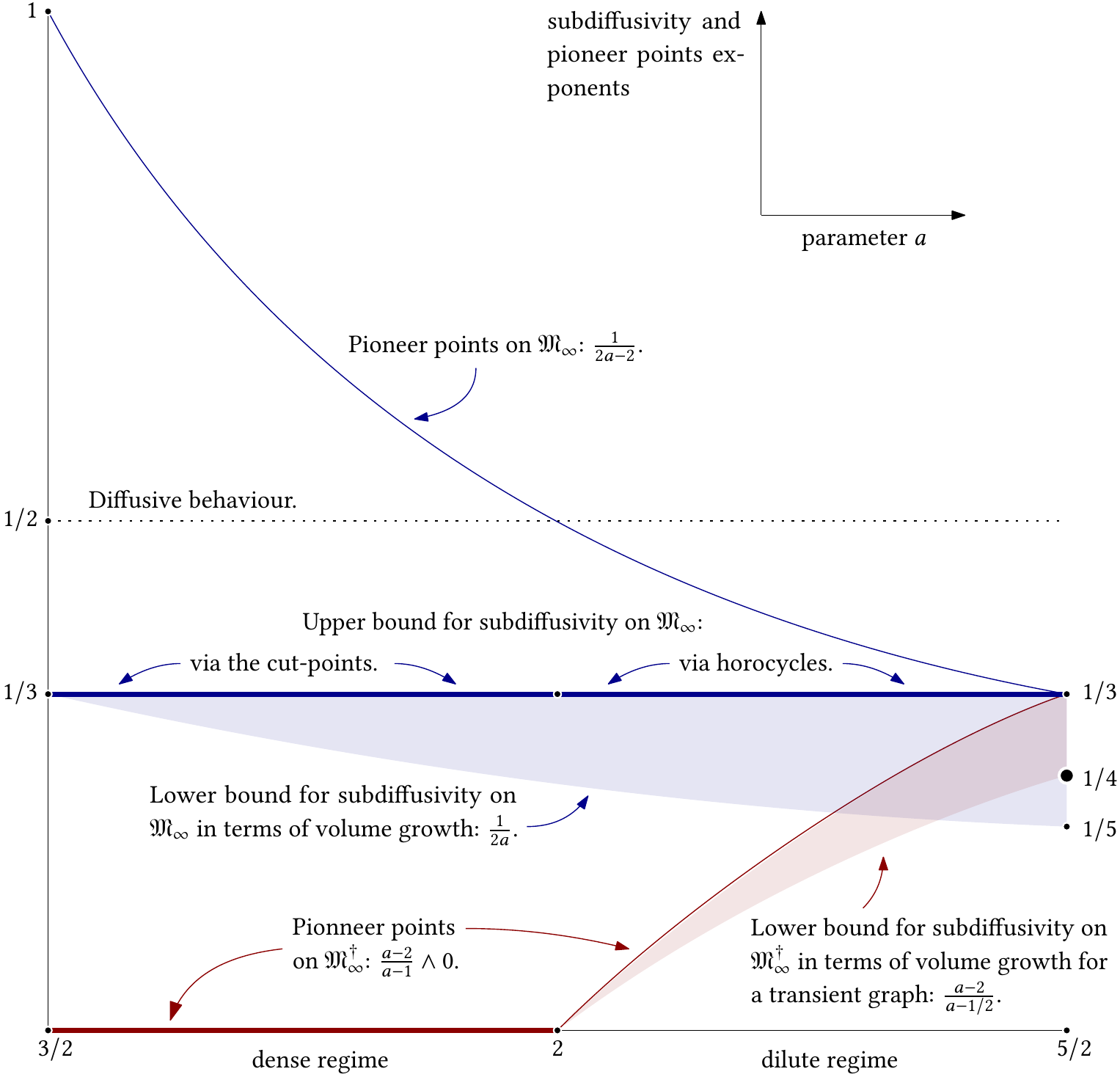}
\caption{A schematic representation of the bounds on the subdiffusivity exponents for the random walk on $\Map_\infty$ in blue (top) and the one on $\Map_\infty^\dagger$ in red (bottom); 
the two thick horizontal lines are the main results of this paper.}
\label{fig:diagramme}
\end{figure}

\paragraph{What is the true subdiffusivity exponent?} 
Let us recapitulate on Figure~\ref{fig:diagramme} the bounds we expect for the subdiffusivity exponent $s_{a}$ relying on properties of the random lattices $ \Map_{\infty}$, some of them being still speculative.  Recall that the volume growth of balls of radius $r$ in $\Map_\infty$  is of order $r^{{2a-1}}$ for $a \in (3/2;5/2]$, see~\cite{Curien-Marzouk:Markovian_explorations_of_random_planar_maps_are_roundish} and of order $ r^{ \frac{a- 1/2}{a-2}}$ in the dual map $ \Map_{\infty}^{\dagger}$ in the dilute phase $a \in (2; 5/2]$, see~\cite{Budd-Curien:Geometry_of_infinite_planar_maps_with_high_degrees}.

Theorem~\ref{thm:sous_diff_primal} provides an upper bound for $s_a \le 1/3$ valid on the primal lattice, in all regimes (thick blue horizontal line on Figure~\ref{fig:diagramme}). On the other hand, a general result in terms of volume growth suggests (see~\cite[Remark p527]{Benjamini-Curien:Simple_random_walk_on_the_uniform_infinite_planar_quadrangulation_subdiffusivity_via_pioneer_points} with the obstacles mentioned in~\cite{Lee:Conformal_growth_rates_and_spectral_geometry_on_distributional_limits_of_graphs}) that $s_a \ge 1/(2a)$. When $a=5/2$ we expect $s_{5/2}= 1/4$ is a broad generality (see Gwynne--Miller and Gwynne--Hutchcroft for the UIPT).

In the case of the dual map, in the dilute regime $a > 2$, recall that we believe that~\eqref{eq:sous_diff_dual} remains valid with balls instead of their hull (dotted thick red line); for a lower bound, if these lattices were transient (which is still an open problem) then the preceding lower bound in terms of the volume growth could be sharpened to $\frac{a-2}{a-1/2}$. This lower bound should be exact at $a=5/2$.

\paragraph{Organisation} 
In Section~\ref{sec:strategie} we first present the global strategy of the proof Theorem~\ref{thm:sous_diff_primal} based on stationarity of the map seen from the random walk and flashing the walk on good subset; then in Section~\ref{sec:peeling} we recall the peeling of planar maps which will be a crucial tool. In Section~\ref{sec:dense} we focus on the dense regime $a < 2$: using the cut points we first prove Theorem~\ref{thm:sous_diff_dual_dense} as well as Theorem~\ref{thm:sous_diff_primal} in this range of values of $a$. Finally Section~\ref{sec:dilue} is devoted to the proof of Theorem~\ref{thm:sous_diff_primal} for all values of $a$.

\paragraph{Acknowledgments} We acknowledge support from the Fondation Math\'{e}matique Jacques Hadamard, the grants \texttt{ERC-2016-STG 716083} ``CombiTop'' and \texttt{ERC 740943} ``GeoBrown'', as well as the grant \texttt{ANR-14-CE25-0014} ``ANR GRAAL''.

\section{Strategy of the proof}
\label{sec:strategie}

\subsection{Subdiffusivity from diffusivity on a sparse subgraph}
The proof of Theorem~\ref{thm:sous_diff_primal} relies on a simple idea (formalised in the following general result) which gives an upper bound on the displacement of a random walk on a random graph by ``flashing it'' on a certain subgraph. Let us denote by $\mathcal{G}$ a random connected (multi-)graph, either finite or infinite, but locally finite in this case, with a distinguished origin vertex $\rho$, and consider a simple random walk $(X_n)_{n \ge 0}$ on $\mathcal{G}$ started at $X_0 = \rho$. Denote by $ \mathcal{B}_{R}$ the ball of radius $R$ around the origin $\rho$ (for the graph distance) in $ \mathcal{G}$.

\begin{lem}\label{lem:conditions_generales_sous_diff}
Let $(\beta_R)_{R \ge 1}$ and $(\gamma_R)_{R \ge 1}$ be two positive sequences and $d \ge 1$. Suppose that for any integer $R \ge 1$, we are given a subset of vertices $\mathcal{G}_R$ of the graph $\mathcal{G}$ such that:
\begin{enumerate}
\item\label{H_degrees} \textsc{(Polynomial growth)} with high probability as $R \to \infty$, the ball $ \mathcal{B}_{R+1}$ has less than $R^{d}$ vertices;

\item\label{H_geom} \textsc{(Geometric Separation)} 
With high probability as $R \to \infty$, a simple random walk on $\mathcal{G}$ started from $\rho$ goes through at least $\beta_R$ different vertices of $\mathcal{G}_R$ before exiting $ \mathcal{B}_{R}$;

\item\label{H_densite} \textsc{(Density)} For every $n \ge 1$, we have that $\P(X_n \in \mathcal{G}_R) \le \gamma_R^{-1}$.
\end{enumerate}
Then with high probability as $R \to \infty$, the random walker $X_i$ belongs to $\mathcal{B}_R$ for every $i \le \gamma_R \beta_R^2 \log^{-7/4} R$.
\end{lem}

The idea of flashing a random walk on a certain subset to deduce subdiffusivity was already used by Kesten~\cite{Kes86} (see also~\cite{Damron-Hanson-Sosoe:Subdiffusivity_of_random_walk_on_the_2D_invasion_percolation_cluster}) where he considered the backbone of the critical infinite incipient percolation cluster on $ \Z^{2}$.

\begin{proof} We shall consider the walk \emph{flashed} on $\mathcal{G}_R$, i.e. the sequence $(Y_i)_{i \ge 1}$ of successive vertices of $\mathcal{G}_R$ visited by the walk. The subset $\mathcal{G}_R$ can be equipped with a connected graph structure induced by $\mathcal{G}$ as follows: two vertices of $\mathcal{G}_R$ are linked by an edge if there exists a path in $\mathcal{G}$ going from one to the other without visiting any other vertex of $\mathcal{G}_R$. By decomposing the probability that $Y$ moves from a vertex to another over all possible corresponding paths for $X$, it is straightforward to show that $Y$ is a (possibly stopped) reversible Markov chain with respect to $\deg_{\mathcal{G}}(\cdot)$, the degree of the vertices in the \emph{original graph} $\mathcal{G}$. The  Varopoulos--Carne bound (see e.g. Lyons \& Peres~\cite[Theorem~13.4]{Lyons-Peres:Probability_on_trees_and_network}) then shows that for $n \ge 1$ and two vertices $y$ and $y'$ in $\mathcal{G}_R$ at distance $d$ \emph{in $\mathcal{G}_R$}, the probability that the flashed walk $Y$ goes from $y$ to $y'$ in exactly $n$ steps is at most
\[2 \left(\frac{\deg_{\mathcal{G}}(y)}{\deg_{\mathcal{G}}(y')}\right)^{1/2} \exp\left(-\frac{d^2}{2n}\right).\]
Recall from Assumption~\ref{H_geom} that the random walker $X$ needs to move for a distance at least $\beta_{R}$ within the graph $ \mathcal{G}_{R}$ in order to escape from $ \mathcal{B}_{R}$. Let us denote by $\sigma_R$ the first instant at which the walk $X$ has made $\beta_R^2 \log^{-3/2} R$ steps in the subset $\mathcal{G}_R$. Summing over all possible starting and ending points inside $ \mathcal{G}_{R} \cap \mathcal{B}_R$ and crudely bounding the degrees by the volume,  we deduce that the probability to move distance $\beta_{R}$ across $ \mathcal{G}_{R}$ before time $\sigma_{R}$ is bounded above by 
\[2 R^{d/2 + d+d}  \exp\left(- \frac{1}{2} \log^{3/2} R\right)  \cv[R] 0,\]
on the event where~\ref{H_degrees} and~\ref{H_geom} are satisfied. On the other hand, by~\ref{H_densite} and Markov's inequality 
\begin{align*}
\Pr{\sigma_R \le \gamma_R \beta_R^2 \log^{-7/4} R} 
&= \Pr{\sum_{k \le \gamma_R \beta_R^2 \log^{-7/4} R} \ind{X_k \in \mathcal{G}_R} \ge \beta_R^2 \log^{-3/2} R}
\\
&\le \beta_R^{-2} \log^{3/2} R \sum_{k \le \gamma_R \beta_R^2 \log^{-7/4} R} \Pr{X_k \in \mathcal{G}_R}
\\
&\le C \log^{-1/4} R \cv[R]0.
\end{align*}
We deduce that with high probability as $R \to \infty$, the random walk $X$ was not able to move across distance $\beta_{R}$ in $ \mathcal{G}_{R}$ within the first $\gamma_R \beta_R^2 \log^{-7/4} R$ steps. A fortiori it could not have escaped $ \mathcal{B}_{R}$ by~\ref{H_geom}.
\end{proof}

The proof of Theorem~\ref{thm:sous_diff_primal} reduces to finding such $\mathcal{G}_R$ which are big enough so $\beta_R$ is large, but not too big so $\gamma_R$ is also large. Indeed, a caricature consists in taking $\mathcal{G}_R$ to be the entire ball of radius $R$, then $\beta_R = R$ but $\gamma_R = 1$ which shows that the walk is at most diffusive; another extreme consists in taking $\mathcal{G}_R$ to be the union of the boundaries of the balls of radius $R$ and $R/2$ which lie at distance $1$ in $\mathcal{G}_R$, but now $\gamma_R$ is quite large and this again would yield a diffusive upper bound in our case.

\subsection{\texorpdfstring{Heuristic for $\mathcal{G}_{R}$}{Heuristic for G\_R}}
\label{sec:heuristic}

\begin{figure}[!ht]\centering
\includegraphics[width=.85\linewidth]{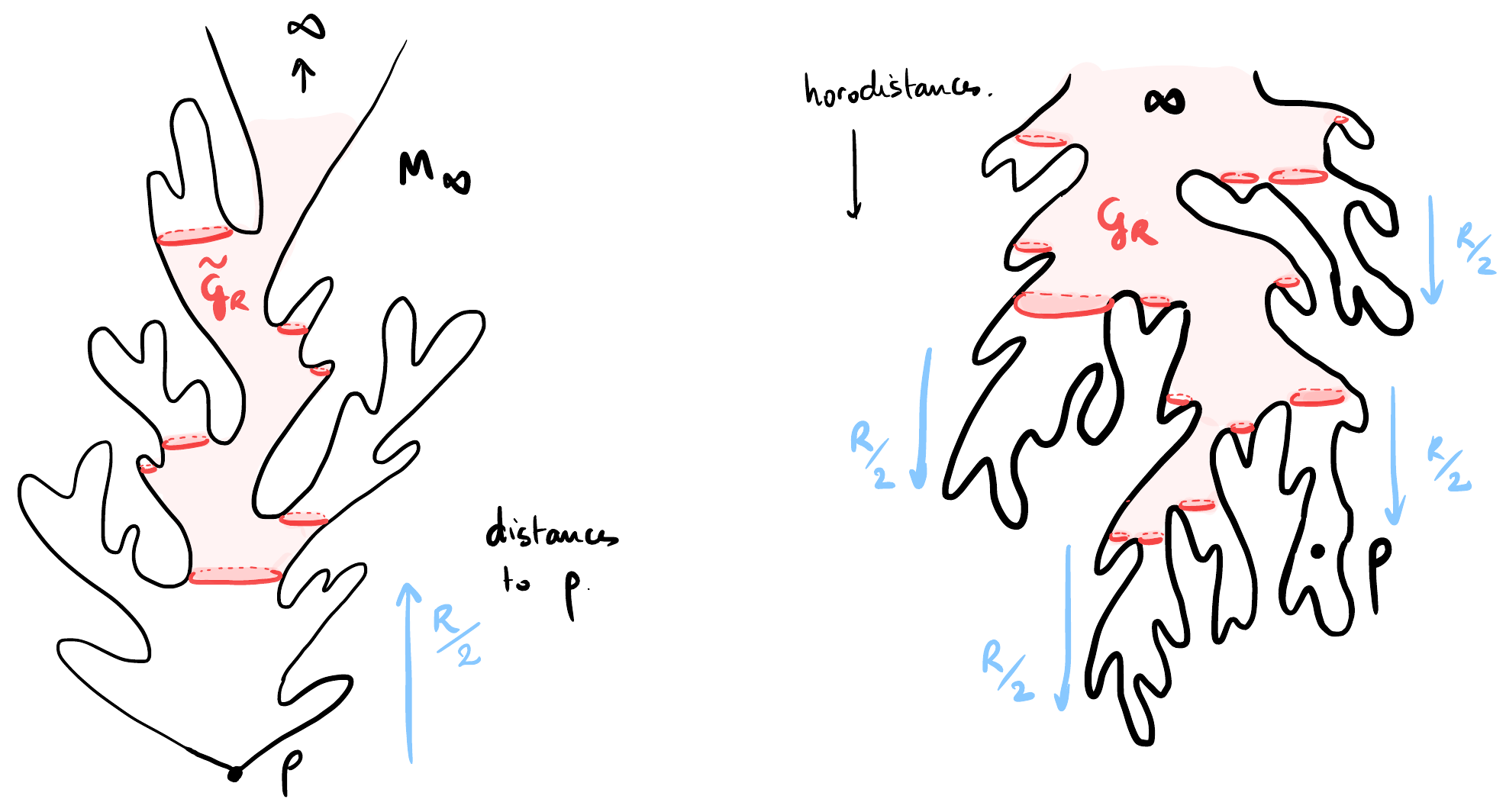}
\caption{A natural try for $ \mathcal{G}_{R}$ on the left (not stationary) and its stationary version using horodistances. The problem is that horodistances are not proved to exist in $\Map_{\infty}$ in general...}
\label{fig:construction1}
\end{figure}
 
Let us give a heuristic of the proof of our main result. A natural guess for $ \mathcal{G}_{R}$ which is thinner than the entire ball but which still necessitate about $R$ (flashed) steps to traverse is the set of vertices  which separates $ \mathcal{B}_{R/2}$ from infinity (see Figure~\ref{fig:construction1} left). The main drawback is that estimating $\P(X_{n} \in \mathcal{G}_{R})$ is a difficult task. This is due to the fact that this set strongly depends on $\rho$. We shall rather construct our random subsets $ \mathcal{G}_{R}$ in a \emph{stationary} way, i.e. such that $\P(X_n \in \mathcal{G}_R) = \P(\rho \in \mathcal{G}_R)$ for all $n$. Since the random graph $\Map_{\infty}$ is itself a stationary random graph~\cite[Proposition~7.9]{Curien:StFlour}, it suffices to construct $ \mathcal{G}_{R}$ in a way that does not depend on the origin $\rho$ of $\Map_{\infty}$. To do this, we shall  use  ``distances from infinity'' or horodistances rather than distances to $\rho$. This horodistance is defined by
\begin{equation}\label{eq:horo}
\ell(u) = \lim_{z \to \infty} \dgr(z, u) - \dgr(z, \rho) \in \Z,
\end{equation}
where $z \to \infty$ means that $z$ escapes from any finite set in the map. We then \emph{define} the set $\mathcal{G}_R$ as those vertices $u$ such that at least $R^{2a-1}$ (the typical volume of a ball of radius $R$) different vertices ``under $u$'', i.e. may be joined to $u$ by a path visiting only vertices with horodistance non-greater than $\ell(u)$ (see Figure~\ref{fig:construction1} right). Since the definition of $ \mathcal{G}_{R}$ does not depend on the origin of the map, it is stationary in the sense that $\P( X_{n} \in \mathcal{G}_{R})$ is constant in $n$. In the notation of Lemma~\ref{lem:conditions_generales_sous_diff} we expect both $\beta_{R} \approx R$ and $\gamma_{R} \approx R$ (see the related Proposition~\ref{prop:proba_arete_good_cas_dilue}) which yields an upper bound of $1/3$ on the subdiffusivity exponent by Lemma~\ref{lem:conditions_generales_sous_diff}. Actually since horodistances~\eqref{eq:horo} are not yet proved to exist in general Boltzmann maps (see~\cite{Curien-Menard-Miermont:A_view_from_infinity_of_the_uniform_infinite_planar_quadrangulation, Curien-Menard:The_skeleton_of_the_UIPT_seen_from_infinity} for the case of the UIPQ and UIPT) we shall use a trick and emulate them on finite maps: we replace horodistances by the distances to an extra large boundary, far away from the root edge, see Section~\ref{sec:dilue}. 
\bigskip

The geometry of $\Map_{\infty}$ undergoes a  phase transition at $a = 2$, and the dense phase $a < 2$ is very different from the dilute phase $a > 2$ (see e.g.~\cite{Budd-Curien:Geometry_of_infinite_planar_maps_with_high_degrees}); in particular, when $a < 2$, the map possesses \emph{cut edges}: large faces touch themselves and disconnect the origin from infinity. In this phase $a < 2$ (Section~\ref{sec:primal_dense}), we actually give another version of a stationary set $ \mathcal{G}_{R}$ as the set of all (vertices adjacent to) edges which separate from infinity a part of the map with volume at least $R^{2a-1}$, see Figure~\ref{fig:cut_points}.
We shall control the number of such cut edges in Proposition~\ref{prop:cut_edges} and this will imply with the preceding notation that $\beta_R \approx R^{4-2a}$. We also evaluate the density of $ \mathcal{G}_{R}$ in Proposition~\ref{prop:sommet_racine_cut_edge} yielding $\gamma_R \approx R^{4a-5}$. This gives $\gamma_R \beta_R^2 \approx R^3$ and proves the same upper bound of $1/3$ on the subdiffusivity exponent. Although this yields the same bound on the subdiffusivity exponent we included this derivation because of the simpler nature of the argument and since the geometric estimates involved are interesting in their own; for example, we recover the recurrence of the walk in this regime, with an explicit lower bound on the effective resistance between the origin and the boundary of the ball of radius $R$, see Remark~\ref{rem:recurrence}.

\begin{figure}[!ht]\centering
\includegraphics[width=.55\linewidth]{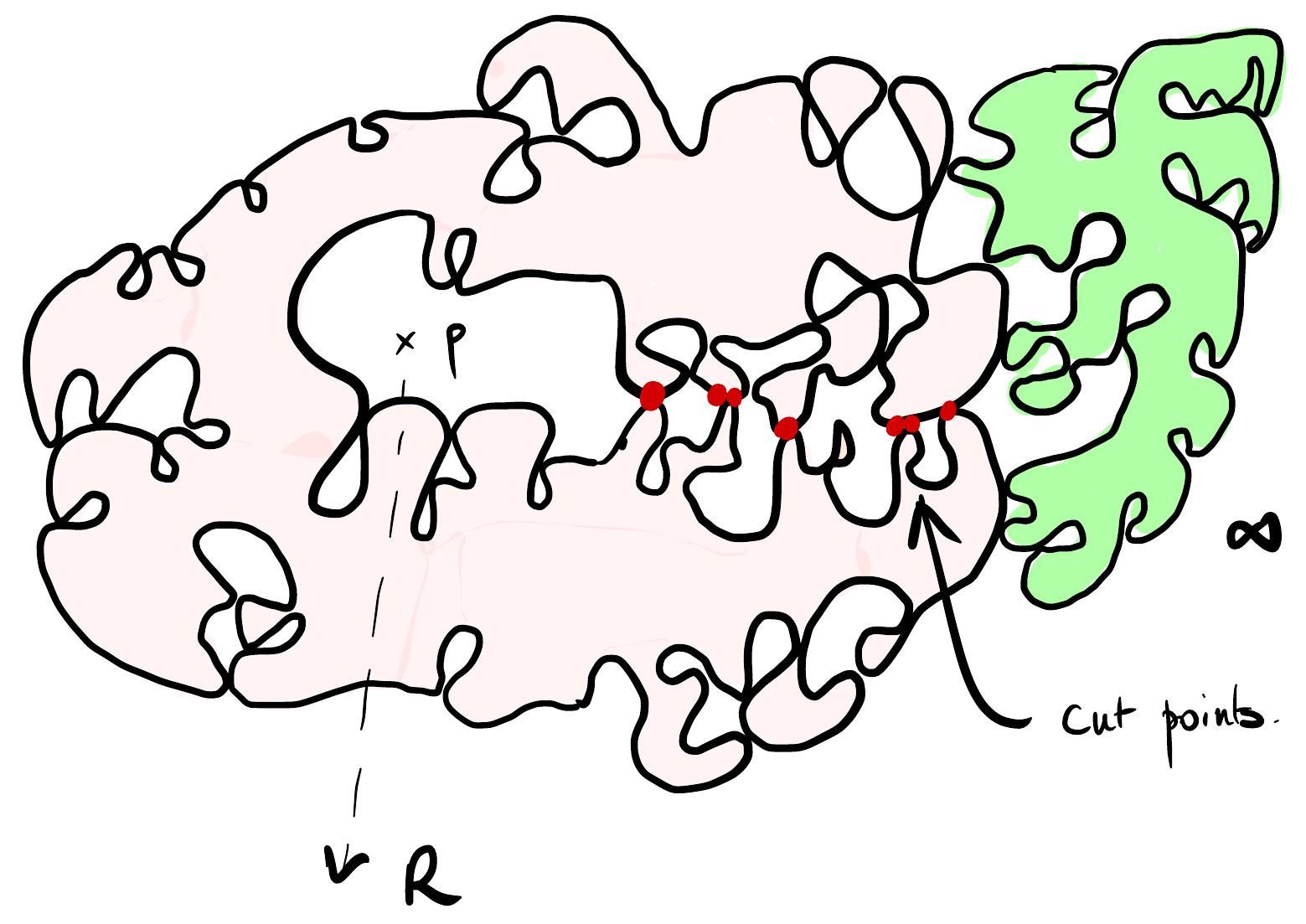}
\caption{Illustration of the cut points in the dense case. We shall prove that roughly $R^{4-2a}$ cut points separate the origin from infinity in the ball of radius $R$. Combining a diffusivity estimate for the random walk flashed on these cut points and the density $R^{5-4a}$ of these cut points yields an upper bound of $1/3$ on the diffusivity exponent.}
\label{fig:cut_points}
\end{figure}

Note that Lemma~\ref{lem:conditions_generales_sous_diff} does not apply to the dual map $\Map_\infty^\dag$ in the dense regime $a < 2$ since the latter has exponential growth~\cite{Budd-Curien:Geometry_of_infinite_planar_maps_with_high_degrees}. In this case, another control on the cut points (Lemma~\ref{lem:cut_face_peeling_layer}) will easily entail Theorem~\ref{thm:sous_diff_dual_dense}.

\section{\texorpdfstring{Peeling of $\Map_\infty$}{Peeling of M\_infinity}}
\label{sec:peeling}
In this section we recall the background of the peeling process on Boltzmann maps and refer to~\cite{Curien:StFlour} for details. We also recall the definition of critical weight sequences of type $a$ as well as the probability measure  $\nu$ which drives the peeling process of Boltzmann maps.

\subsection{\texorpdfstring{Weight sequences of type $a \in (3/2;5/2]$}{Weight sequences of type a in (3/2;5/2]}}
\label{sec:nu}

As soon as the weight sequence $ \q$ is fixed and admissible, we denote by $W^{(\ell)}$ the total $\q$-weight of all finite maps with a root face of degree $2 \ell$. Under those assumptions, a very general enumeration result, see~\cite[Lemma~3.13]{Curien:StFlour} gives a ``strong ratio limit'' theorem, in the sense that the ratio $W^{(\ell+1)} / W^{(\ell)}$ converges to some explicit constant $c_{ \q}>1$ when $\ell \to \infty$.
These numbers, together with the weight sequence $\q$ enable us to define a probability distribution $\nu$ (see~\cite[Lemma~5.2]{Curien:StFlour}) by
\begin{equation}\label{eq:nu}
\nu(k) = \left\{ \begin{array}{ll} q_{k+1} c_{ \q}^{k}  & \text{for } k \ge 0 \\
2W^{(-1-k)} c_{ \q}^{k} & \text{for } k \le -1. \end{array} \right.
\end{equation}
The criticality condition on $ \q$ is then equivalent to the fact that a $\nu$-random walk oscillates (see~\cite{Budd:The_peeling_process_of_infinite_Boltzmann_planar_maps} and~\cite[Theorem~5.4]{Curien:StFlour}). The renewal function $h^{\uparrow}$ of this walk is then universal (it does not depend on $\q$ once it is critical) and is equal to 
\begin{equation}\label{eq:h_up}
h^{\uparrow} (k) =  2k \cdot 2^{-2k} \binom{2 k}{k} \qquad\text{for } k \ge1.
\end{equation}
Furthermore, the weight sequence is of type $a \in (3/2;5/2)$ if and only if, as $k \to \infty$,
\begin{equation}\label{eq:tailnu}
\nu(-k) \sim \mathsf{p}_{\q} \cdot k^{-a} \qquad\text{and}\qquad \nu([k,\infty)) \sim \frac{\mathsf{p}_{\q}}{a-1} \cos(a \pi) \cdot k^{-a+1},
\end{equation}
where $\mathsf{p}_{\q}$ is some constant which depends on our weight sequence $\q$.  Also $\q$ is of type $a=5/2$ if  $\nu(-k) \sim \mathsf{p}_{\q} \cdot k^{-5/2}$ and $\nu([k,\infty)) = o( k^{-3/2})$ as $k \to \infty$. We refer to~\cite[Chapter~V]{Curien:StFlour} for details.

\subsection{\texorpdfstring{Filled-in peeling of $\Map_\infty$}{Filled-in peeling of M\_infinity}}

In this section, we briefly recall the filled-in peeling process of $\Map_\infty$ and refer the reader to~\cite{Curien:StFlour} for details. This will be our key tool in order to prove the intermediate results mentioned above. \medskip 

We shall use the root transformation, see~\cite[Figure~3.2]{Curien:StFlour}, to see any bipartite planar map as a map with a root face of degree $2$ after splitting the root edge. 
A \emph{submap} $\e$ with a unique hole of a given map $\map$ is a map with a distinguished face with a simple boundary (called its \emph{hole}), such that $\map$ can be recovered by gluing a proper map with (general) boundary inside the hole of $\e$. A filled-in peeling process of $\Map_\infty$ is a sequence of submaps $ \e_{0} \subset \e_{1} \subset \cdots \subset \Map_{\infty}$ constructed recursively started from $\e_{0} $ being simply the $2$-gon containing the root edge in the following way. At each step $n$, we select an edge $ \mathcal{A}( \e_{n})$ (the \emph{peel edge}) on the boundary of the hole of $ \e_{n}$ and aim at revealing its ``hidden'' side; two cases may appear, as illustrated in Figure~\ref{fig:peeling}. We denote by $\ell$ the half-perimeter of the hole of $\e_n$.
\begin{itemize} 
\item Either the peel edge is incident to a new face in $ \Map_{\infty}$ of degree $2k$, then $ \e_{n+1}$ is obtained from $ \e _{n}$ by gluing this face on the peel edge without performing any other identification. This event is called event of type $ \mathsf{C}_{k}$ and appear with probability 
\[\Pr{\mathsf{C}_k} =  \frac{h^\uparrow(\ell+k-1)}{h^\uparrow(\ell)} \nu(k-1).\]
\item Or the peel edge is incident to another face of $ \e_{n}$ in the map $\Map_{\infty}$, in which case we perform the identification of the two boundary edges of $\e_n$. When doing so, the hole  of $ \e_{n}$ of perimeter, say $2\ell$, is split into two holes of perimeter $2\ell_{1}$ and $2 \ell_{2}$ with $\ell_{1}+\ell_{2}=\ell-1$. Since $\Map_{\infty}$ is one-ended almost surely, only one of these holes contains an infinite region in $\Map_{\infty}$. We then fill-in the finite hole with the corresponding map inside $ \Map_{\infty}$ to obtain $  \e_{n+1}$. We speak of event of type $ \mathsf{G}_{*,\ell_{1}}$ or $ \mathsf{G}_{\ell_{2}, *}$ depending whether the finite hole is on the left or on the right of the peel edge and they happen with probability 
\[\Pr{\mathsf{G}_{*,k}} = \Pr{\mathsf{G}_{k,*}} =  \frac{1}{2}\frac{h^\uparrow(k)}{h^\uparrow(\ell)} \nu(k-\ell).\]
\end{itemize}

\begin{figure}[!ht]\centering
\includegraphics[width=.85\linewidth]{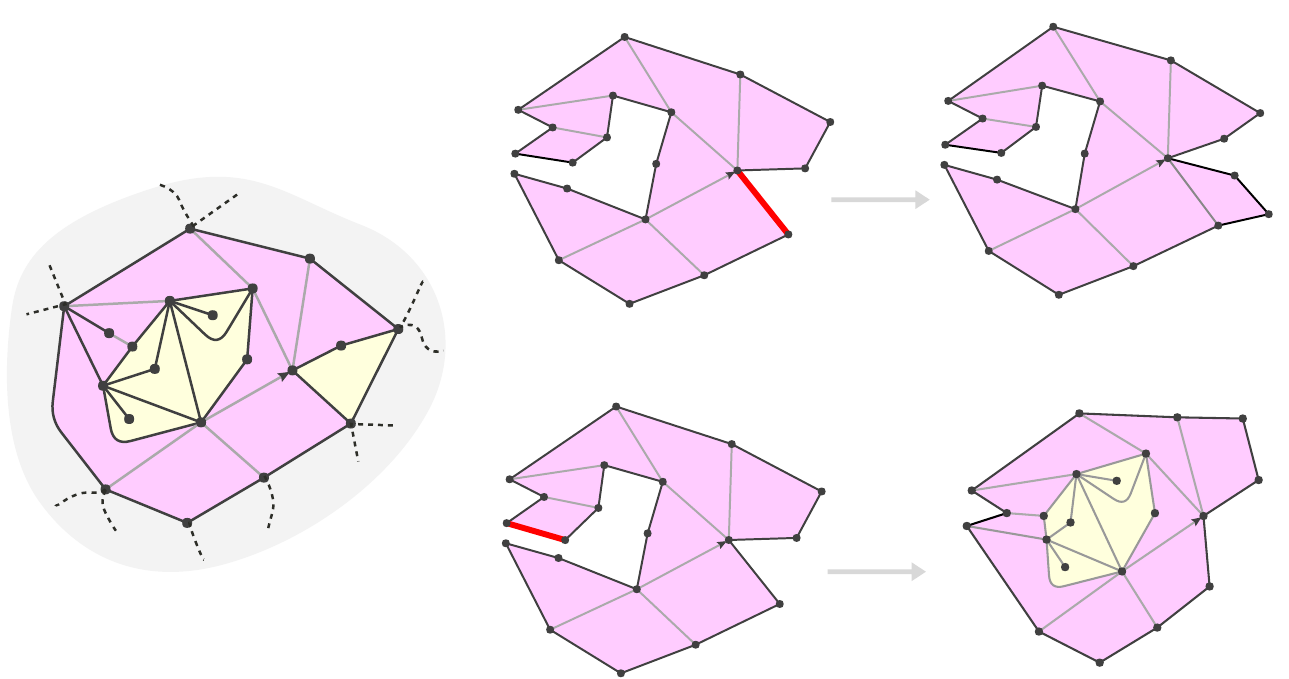}
\caption{Illustration of the filled-in peeling process. In the left-most Figure~we have explored a certain region $ \e_{n} \subset \Map_{\infty}$ corresponding to the faces in pink glued by the edges in gray. Depending on the edge to peel at the next step we may end-up either with an event of type $ \mathsf{C}_{2}$ (top figures), or an event of type $ \mathsf{G}_{3, *}$ (bottom figures).}
\label{fig:peeling}
\end{figure}

Let us stress that the choice of the peel edge at each step is given by a peeling algorithm $\mathcal{A}$ which may depend on another source of randomness as long as it is independent of the unrevealed part. 
In the next subsection, we shall describe a particular algorithm designed to reveal the hull of the balls one after the other. Another property that we shall use is the \emph{spatial Markov property} which says that for any time $n$, the map filling-in the hole of $\e_n$ is independent of $\e_n$ and is distributed as $\Map_\infty^{(\ell)}$ the infinite Boltzmann map of the plane with a root face of degree $2\ell$.

\subsection{Applications: Peeling by layers}
\label{sec:peeling_layer}

As a direct consequence of the peeling of $\Map_{\infty}$ we can compute the degree of the root face. Recall that in the above presentation, we used the root transformation~\cite[Figure~3.2]{Curien:StFlour} to see any bipartite planar map as a map with a root face of degree $2$ after splitting the root edge. After gluing back the two sides of this $2$-gon together, the law of the degree of the root face of $ \Map_{\infty}$ is given by the perimeter of the first face we reveal during the peeling process i.e. 
\begin{equation}\label{eq:rootdegree}
\Pr{\deg( \rootface)=2k} =   \frac{h^\uparrow(k)}{h^\uparrow(1)} \nu(k-1), \qquad \text{for } k \ge 1.
\end{equation}
Let us describe two peeling algorithms we will use later.

\subsubsection{Peeling by layers on the dual}
The peeling by layers on the dual, algorithm $\mathcal{A}_{\mathrm{dual}}$, is designed to reveal the hull of the dual balls centred at the root face one after the other. First, as in any peeling, set $\e_0$ to be a $2$-gon which serves as the root face. The algorithm  $\mathcal{A}_{\mathrm{dual}}$ will then ``turn around'' the boundary and peel at edges adjacent to a face whose dual graph distance to the root face is minimal. If $\theta_{r}$ is the first time at which no edge is adjacent to a face at dual distance $r$ from the root face then the piece revealed is equal to the hull of the ball\footnote{We mean here the map obtained by keeping only the faces  that are at dual distance less than or equal to $r$ from the root face and cutting along all the edges which are adjacent on both sides to faces at dual distance $r$ from $\rootface$.} of radius $r$ in $ \Map_{\infty}^{\dagger}$. See Figure~\ref{fig:peelinglayers} for an illustration and~\cite[Chapter~13.2]{Curien:StFlour} for details.

\begin{figure}[h!]\centering
\includegraphics[width=.6\linewidth]{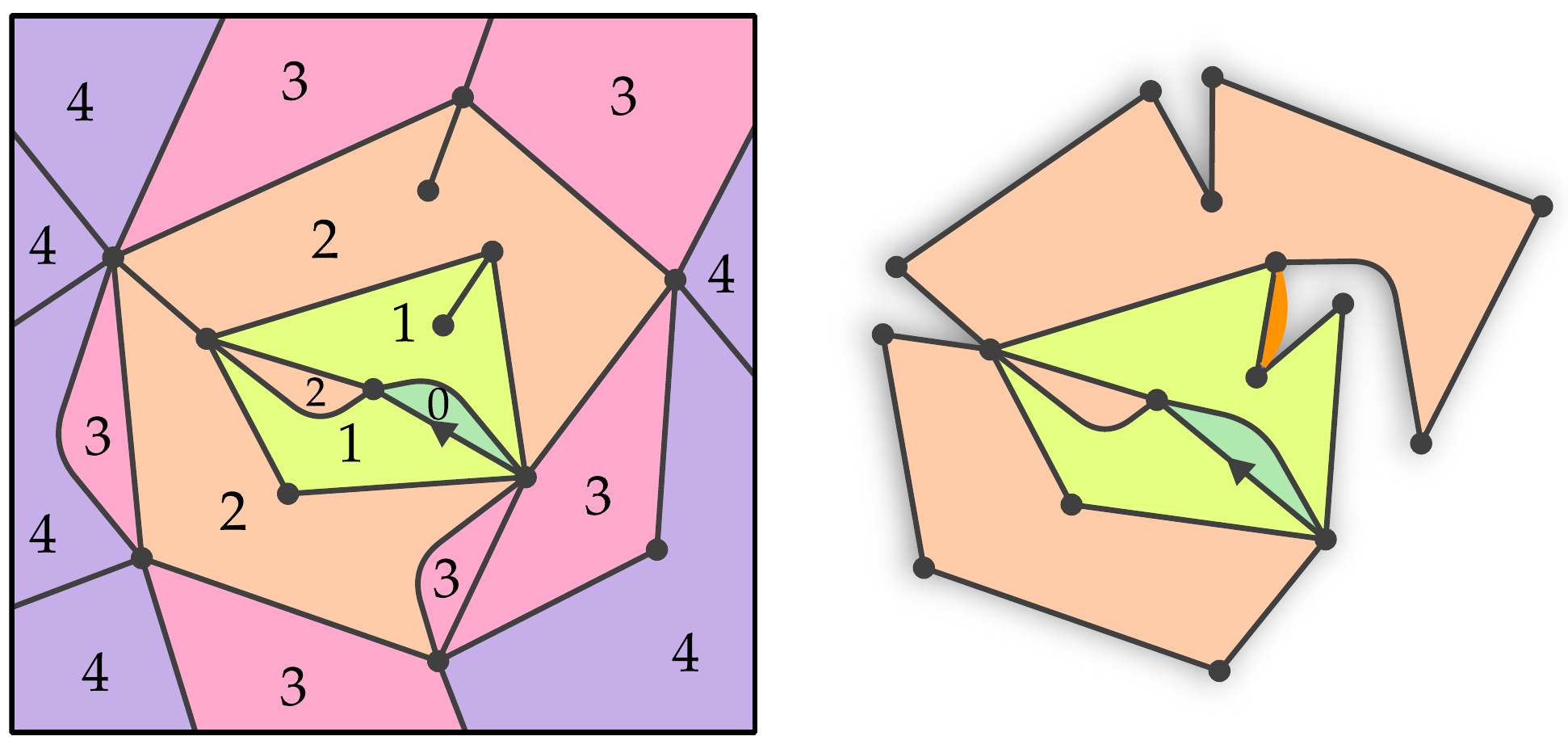}
\caption{(from~\cite{Curien:StFlour})  The left Figure~shows a portion of an infinite planar map with faces labelled according to (dual) graph distance to the root face. The submap on the right depicts a possible state of the peeling by layers. The next edge to peel is indicated in orange.}
\label{fig:peelinglayers}
\end{figure}

\subsection{Peeling by layers on the primal}

We shall also need the following algorithm $\mathcal{A}_{\mathrm{metric}}$, illustrated in Figure~\ref{fig:ametricquad}, which discovers one after the other the hull of the balls of the original map centred at the origin $\rho$ of the root edge. This is done as above by turning around the boundary of the explored maps and always peeling at edge $\mathcal{A}_{\mathrm{metric}}( \e)$ whose right end point minimises the distance (for the primal graph distance inside $\e$) to the origin vertex of the map (if there are several choices, we break the ties deterministically). The main difference with $ \mathcal{A}_{ \mathrm{dual}}$ is that the distances of the vertices along the boundary of $\e_{i}$ to the origin may differ in $\e_{i}$ and in $\Map_{\infty}$. However, it is easy to check that they agree for those vertices at minimal graph distance from the origin.

\begin{figure}[!ht]\centering
\includegraphics[width=0.8\linewidth]{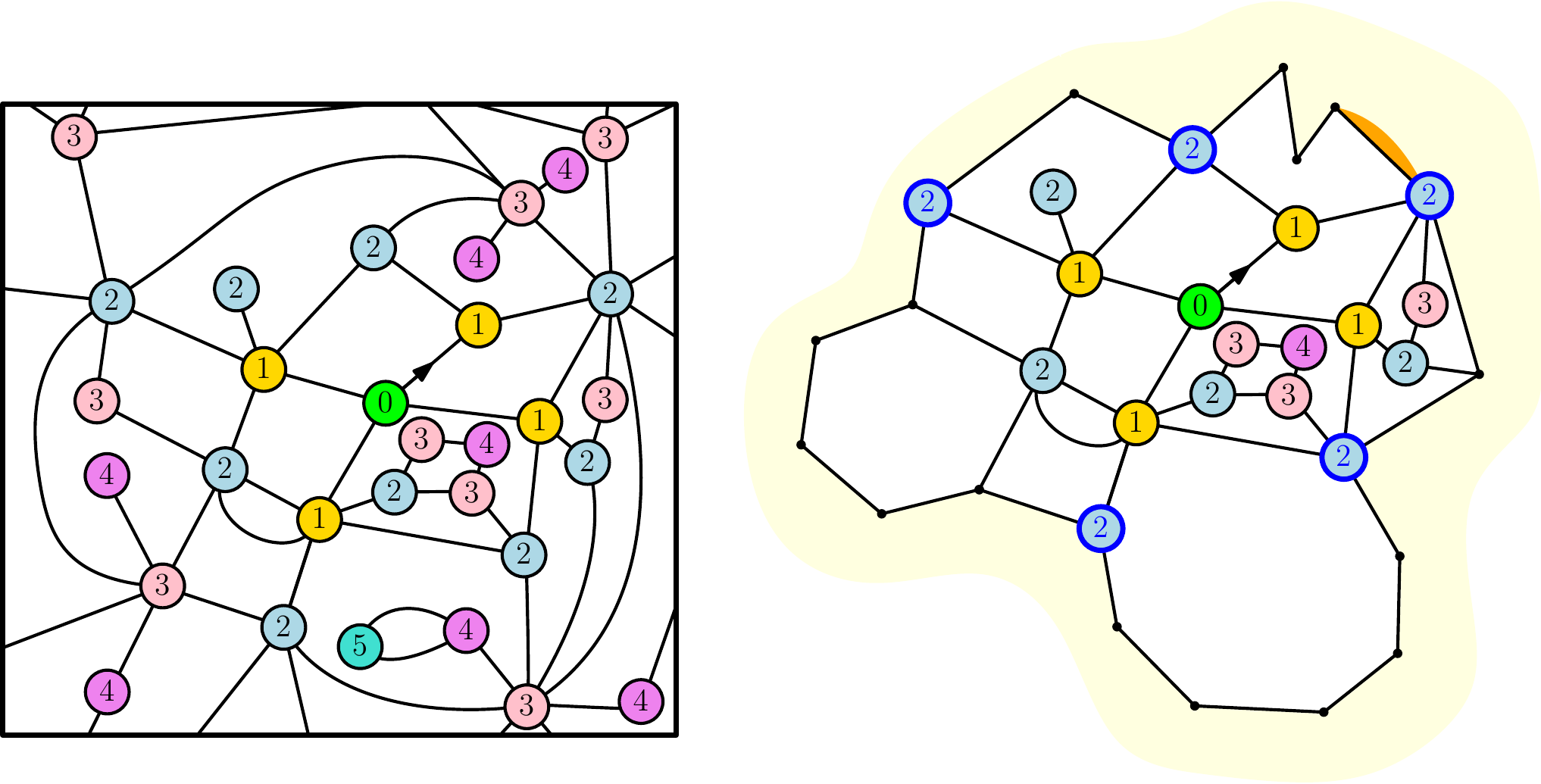}
\caption{Illustration of the algorithm $ \mathcal{A}_{ \mathrm{metric}}$. The labelling represents distances of the vertices to the origin. On the right a current state $\e_i$ of the exploration. Notice that the vertices on the boundary with minimal distances to the origin have the same labels in $\e_{i}$ and in the underlying map. We always peel at an edge whose right end point minimises this distance (inside $\e_{i}$).}
\label{fig:ametricquad}
\end{figure}
 
Applying the above algorithm to discover the $1$-neighbourhood of the origin in the map, it is easy to prove that the vertex degree of the origin in $\Map_{\infty}$ has an exponential tail (compared with~\eqref{eq:rootdegree} in the dual map), that is 
\begin{equation}\label{eq:rootdegreevertex}
\Pr{\deg(\rho) \ge k} \le \ex^{- c k}, \qquad k \ge 0,
\end{equation}
for some constant $c>0$ (depending on the critical weight sequence $ \q$). See~\cite[Lemma~15.7]{Curien:StFlour}, also~\cite[Section~4.1]{Bjornberg-Stefansson:Recurrence_of_bipartite_planar_maps} or~\cite[Theorem~7.1]{Stephenson:Local_convergence_of_large_critical_multi_type_Galton_Watson_trees_and_applications_to_random_maps} for a proof based on a Schaeffer-type construction.

\section{Subdiffusivity via cut points in the dense phase}
\label{sec:dense}

We focus in this section on the dense regime $a \in (\frac{3}{2}, 2)$ which is simpler than the dilute regime $a \in (2, \frac{5}{2})$ because of the existence of \emph{cut points} for these lattices: A cut edge in $\Map_{\infty}$ is an edge whose removal disconnects the maps into two parts; if the origin belongs to the finite one, then the cut edge has to be traversed by any infinite path starting from the origin. In the case of $\Map_{\infty}^{\dagger}$ we should consider the cut faces, which are faces of $\Map_{\infty}$ whose removal disconnects the root face from $\infty$ for the dual graph $\Map_{\infty}^{\dagger}$.

\subsection{The dual map}
\label{sec:dual_dense}

Let us start with the dual map $\Map_\infty^\dag$ for which our proof of Theorem~\ref{thm:sous_diff_dual_dense} is quite simple. The main technical ingredient is the following lemma which is based on results of~\cite[Section~5]{Budd-Curien:Geometry_of_infinite_planar_maps_with_high_degrees}. Recall that $\rootface$ denotes the root face of the map.

\begin{lem}\label{lem:cut_face_peeling_layer}
Fix $a \in (\frac{3}{2},2)$.
There exist $c_a \in (0, \infty)$, $\eta \in (0,1) $ and $R_0 \ge 1$ such that for all $R \ge R_0$ and all $k \ge 1$, in $\Map_\infty$ we have
\[\Pr{\text{there is no cut face of degree} \ge \ex^{c_a R} \text{ at dual distance} \le kR \text{ from } \rootface} \le \eta^k.\]
\end{lem}

\begin{proof}
Let us perform the peeling by layer on the dual of $\Map_\infty$, i.e. with the algorithm $\mathcal{A}_{\mathrm{dual}}$; recall that for every $k \ge 1$, we denote by $\theta_k$ the least time $i \ge 1$ such that the peeling process at time $i$ has entirely revealed the hull of radius $k$ in the dual map.
Results of~\cite[Section~5]{Budd-Curien:Geometry_of_infinite_planar_maps_with_high_degrees} show that the perimeter of the (hulls of the) balls in $\Map_{\infty}^{\dagger}$ grow exponentially fast. In particular, there exists $c_{a}>0$ such that for $R$ large enough, the perimeter at time $\theta_{R-3}$ is larger than $\ex^{R \mathsf{c_a}}$ with probability at least $1/2$. On this event, in the discussion closing Section~5.2.1 in~\cite{Budd-Curien:Geometry_of_infinite_planar_maps_with_high_degrees}, it is further shown that with a probability uniformly bounded below, there exists a cut face in the hull of radius $R$  for the dual graph distance (i.e.~in the within the next three turns of the peeling algorithm) with degree at least $\ex^{R \mathsf{c_a}}$. In a few words, if one continues the peeling exploration after $\theta_{R-3}$, then there is a probability bounded below that within the next turn we discover a large face of degree proportional to the perimeter, and further that this face will create a cut face when two edges of this face are identified in a $ \mathsf{G}_{*,*}$ event in a way that separates the origin from $\infty$. This discussion shows the case $k=1$ of the proposition.

To get the full statement we just use the spatial Markov property: after time $\theta_{R}$ if $\ell_{R}$ is the half-perimeter of the hole, then the remaining random map to explored is distributed as $\Map_{\infty}^{(\ell_{R})}$. The arguments in~\cite[Section~5]{Budd-Curien:Geometry_of_infinite_planar_maps_with_high_degrees} show that above discussion holds for $\Map_{\infty}^{(\ell)}$ instead of $\Map_{\infty}$: for $R$ large enough we have 
\[\Pr{ \text{in } \Map_{\infty}^{(\ell)} \text{ there is no cut face of degree} \ge \ex^{c_a R} \text{ at dual distance} \le R \text{ from } \rootface} \le \eta,\]
where $\eta <1$ does not depend on $\ell$. The statement of the proposition then follows by exploring up to distances $R,2R,3R, \ldots,  kR$ and combining the Markov property with the above display.
\end{proof}

With this lemma at hand, let us prove Theorem~\ref{thm:sous_diff_dual_dense} on the random walk on $\Map_\infty^\dag$; the argument is depicted in Figure~\ref{fig:proofdensedual}.

\begin{proof}[Proof of Theorem~\ref{thm:sous_diff_dual_dense}]
Fix $ \varepsilon>0$. Our goal is to see that with probability at least $ 1- \varepsilon$, within the first $ \ex^{ R}$ steps of the walk on $\Map_\infty^\dagger$ we do not escape from the dual ball of radius $\delta R^2$ for some $\delta > 0$. Let us first look at the degrees of faces (i.e. vertices of $\Map_\infty^\dagger$) we encounter during this journey:
For any $m \ge 1$,
\[\Pr{ \sup_{0 \le i < \ex^R} \deg(X_i^\dagger) \ge m}
\le \ex^R \cdot \Pr{\deg(X_0^\dagger) \ge m}
\le \ex^R \cdot \mathrm{Cst} \cdot  m^{3/2-a},\]
where the first inequality follows from a union bound and the stationarity of the walk, whilst the second follows from~\eqref{eq:rootdegree} and~\eqref{eq:tailnu} and the fact that $h^\uparrow(k) \le 2\sqrt{k}$ for every $k \ge 1$. Taking $m = \ex^{ \frac{2}{ a-3/2}R}$ we deduce that with high probability, the walk does not visit any face of degree larger than $m$ with high probability during the first $\ex^R$ steps. In Lemma~\ref{lem:cut_face_peeling_layer} we put $R\equiv\frac{2}{(a-3/2) c_a}R$ and take $k=k_0$ large enough  so that $\eta^{k_0} \le \varepsilon$, we deduce that for $R$ large enough we have
\[\Pr{\text{there is a cut face of degree} \ge m \text{ at dual distance} \le \frac{2 k_0}{(a-3/2) c_a} R \text{ from } \rootface} \ge 1-\varepsilon.\]
Combining these two findings we already deduce that with high probability, the walk cannot visit such a cut face in the first $\ex^R$ steps and is thus confined in the hull of the ball of radius $\frac{2 k_0}{(a-3/2) c_a}R$ with high probability. However, the dual distances it could reach within this hull could a priori be large. To control them, we choose $\delta >0$ so that $ \ex^2 \eta^\delta < 1$ and put $k= \delta R$ in Lemma~\ref{lem:cut_face_peeling_layer} to deduce that 
\[\Pr{\text{there is no cut face of degree} \ge m \text{ at dual distance} \le \delta R^2 \text{ from } \rootface} 
\le \ex^{- 2R}.\]
By the union bound and stationarity, we deduce that with high probability, during the first $\ex^R$ steps of the walk, we are always able to find a cut face of degree $\ge m$ within distance $\delta R^2$ of the current state. Since we know that we can find such a face a distance $\frac{2 k_0}{(a-3/2) c_a}R$ from the origin, this implies that the walk cannot have reached distance more than $\delta R^2 +  \frac{2 k_0}{(a-3/2) c_a}R$ from the origin, with high probability.
\end{proof}

\begin{figure}[!ht]\centering
\includegraphics[width=8cm]{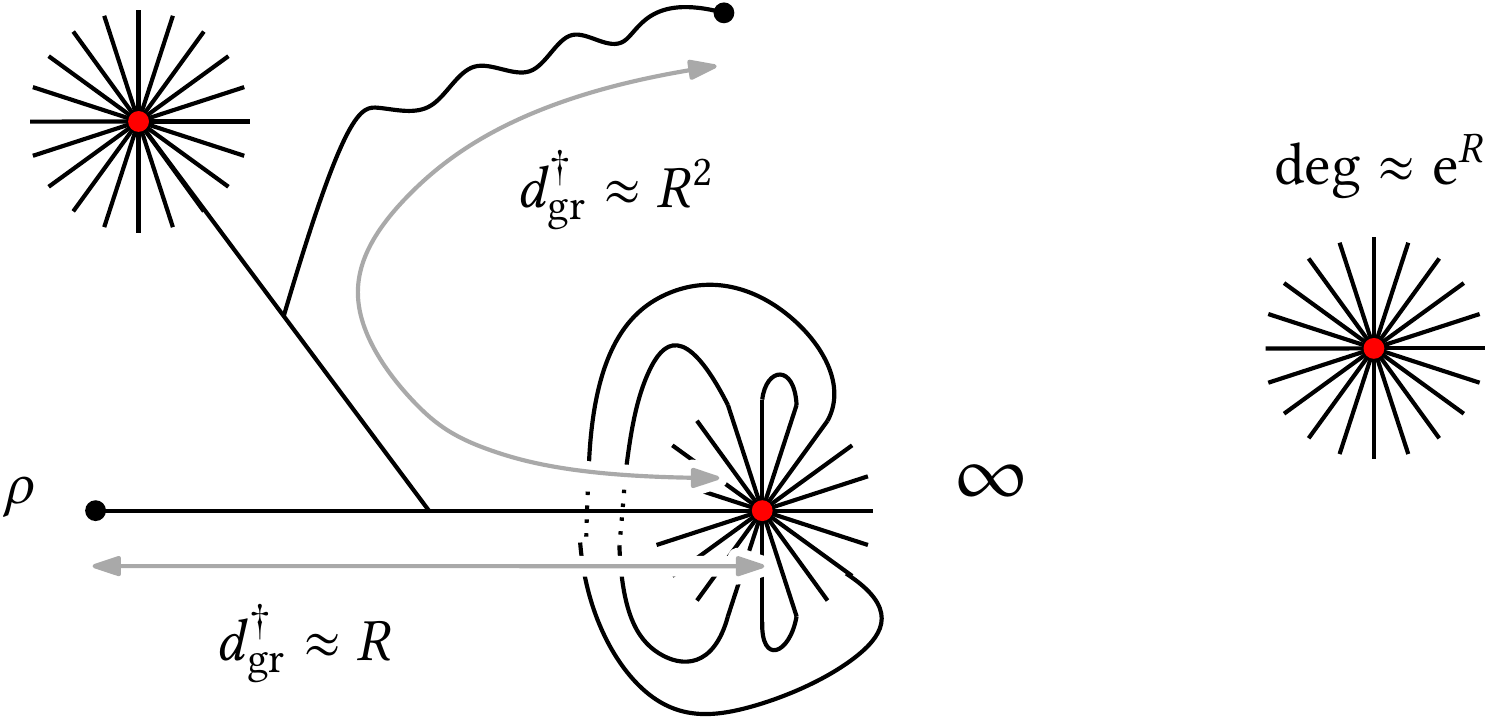}
\caption{Illustration of the proof of Theorem~\ref{thm:sous_diff_dual_dense}. Vertices of degree larger that $\ex^{ \frac{2}{ a-3/2}R}$ are in red and we know that the walk cannot step on them within the first $\ex^{R}$ steps with high probability. Besides, with high probability, one of these vertices is a cut face at distance $\approx R$ from the root face, and for any time $0 \le i < \ex^{R}$ we can find such a cut face at distance less than $\approx R^{2}$ from the current state at time $i$. We deduce that we cannot escape further away than $\approx R^{2}$ from the origin.}
\label{fig:proofdensedual}
\end{figure}

\begin{rem}
Theorem~\ref{thm:sous_diff_dual_dense} shows a $\log^2 n$ upper bound for the displacement of the walk on $\Map_\infty^\dag$ up to time $n$ in the dense regime $a < 2$. Since balls exhibit an exponential volume growth~\cite{Budd-Curien:Geometry_of_infinite_planar_maps_with_high_degrees} we believe that this displacement grows in fact like some constant times $\log n$.
\end{rem}

\subsection{Estimates on cut edges on the primal map}
\label{sec:primal_dense}

Let us next focus on the primal map, still in the dense regime $a < 2$. Recall that $\Map_\infty$ has polynomial volume growth~\cite[Proposition~2]{Curien-Marzouk:Markovian_explorations_of_random_planar_maps_are_roundish}; we shall prove Theorem~\ref{thm:sous_diff_primal} by relying on Lemma~\ref{lem:conditions_generales_sous_diff}. An edge $e$ of $\Map_{\infty}$ is called an \emph{$R$-cut edge} if it separates from infinity a part of the map of volume (e.g. the number of edges, but it could be the number of vertices or faces) at least $R^{2a-1}$. As alluded after Lemma~\ref{lem:conditions_generales_sous_diff}, we shall consider the set $\mathcal{C}_R$ made of all the extremities of $R$-cut edges, the set $\mathcal{G}_R$ in Lemma~\ref{lem:conditions_generales_sous_diff} shall be taken as $\mathcal{C}_{cR}$ for some well chosen $c$.
The next result bounds the density of $R$-cut edges in the map.

\begin{lem}
\label{lem:racine_cut_edge}
There exist two constants $0 < c < C < \infty$ such that for every $R \ge 1$,
\[c R^{5-4a} \le \Pr{\text{the root edge is an } R \text{-cut edge}} \le C R^{5-4a}.\]
\end{lem}

\begin{proof}
Let us denote by $P \ge 0$ the half-perimeter and by $V \ge 0$ the volume of the finite map separated from infinity by the root edge (this map could be reduced to the vertex map). We first claim that for $k \ge 1$ we have
\begin{equation}\label{lawperim}
c_{1} k^{3/2-2a} \le \Pr{P=k} \le c_{2} k^{3/2-2a}
\end{equation}
for some $c_{1},c_{2} >0$.  Indeed, $P=k$ if the following occurs when peeling $\Map_\infty$: We start with $\e_0$ being a digon obtained by opening up the root edge and we peel both sides one after the other; first we discover a large face, with degree, say, $2\ell \ge 2k+2$, and then, at the second step, the peel edge gets identified with another edge on this large face, and swallows a part of length $2k$ containing the finite part of the map. 
From the transition probabilities of the peeling recalled in Section~\ref{sec:peeling}, this occurs with probability
\[\sum_{\ell \ge k+1} \nu(\ell-1) \frac{h^\uparrow(\ell)}{h^\uparrow(1)} \cdot \frac{1}{2} \nu(-k-1) \frac{h^\uparrow(\ell-k-1)}{h^\uparrow(\ell)}
= \frac{1}{2} \frac{\nu(-k-1)}{h^\uparrow(1)} \sum_{\ell \ge k+1} \nu(\ell-1) h^\uparrow(\ell-k-1).\]
Recall that $k^{a} \nu(-k)$, as well as $k^{a-1}\nu([k, \infty))$ and $k^{-1/2} h^\uparrow(k)$ all converge to positive and finite limits, so the preceding display is bounded above by some constant times
\begin{align*}
k^{-a} \sum_{\ell \ge k} \nu(\ell) \sqrt{\ell}
&= k^{-a} \sum_{K \ge 0} \sum_{\ell = 2^K k}^{2^{K+1} k} \nu(\ell) \sqrt{\ell}
\\
&\le k^{-a} \sum_{K \ge 0} \nu([2^K k, \infty)) \sqrt{2^{K+1} k}
\\
&\le C k^{3/2 - 2a} \sum_{K \ge 0} (2^{K+1})^{3/2 - a},
\end{align*}
for some $C > 0$, and a similar lower bound holds. The last series converges since $a > 3/2$.

Now recall from the spatial Markov property that once such an identification is made, the map which fills-in the finite hole in the peeling is independent and has the law of a finite Boltzmann map $\Map^{(k)}$ with perimeter $2k$; according 
to~\cite[Proposition~10.4]{Curien:StFlour} (see also~\cite[Proposition~3.4]{Budd-Curien:Geometry_of_infinite_planar_maps_with_high_degrees} for the number of vertices as notion of volume)
admits the following scaling limit: $k^{-(a-1/2)} |\Map^{(k)}|$ converges in distribution to a non degenerate random variable of support $ \R_{+}$ as $k \to \infty$. We deduce that
\[\Pr{V \ge R^{2a-1}} \ge c \cdot \Pr{P \ge R^2} \underset{\eqref{lawperim}}{\ge} c' R^{5-4a}\]
for some constants $c,c' >0$. For the upper bound we also need to consider the case $P < R^2$.
Proposition~10.4 in~\cite{Curien:StFlour} (or Proposition~3.4 in~\cite{Budd-Curien:Geometry_of_infinite_planar_maps_with_high_degrees}) also proves that $k^{-(a-1/2)} \E[|\Map^{(k)}|]$ converges to some non degenerate constant as $k \to \infty$ so we deduce from Markov's inequality that 
\[\Pr{V \ge R^{2a-1}} \le \Pr{P \ge R^{2}} + \Es{P^{a-1/2} \ind{P \le R^{2}}} R^{-(2a-1)} \underset{\eqref{lawperim}}{\le} C R^{5-4a}\]
for some $C>0$.
\end{proof}

Recall that we are interested in the set $\mathcal{C}_R$ made of all the \emph{extremities} of $R$-cut edges. By stationarity, the probability that $X_n$ belongs to $\mathcal{C}_R$ does not depend on $n$ and is equal to the probability that the origin $\rho$ of the root edge belongs to $\mathcal{C}_R$. The next result provides Assumption~\ref{H_densite} of Lemma~\ref{lem:conditions_generales_sous_diff}.

\begin{prop}[Density of $\mathcal{C}_{R}$]
\label{prop:sommet_racine_cut_edge}
There exist two constants $0 < c_1 < c_2 < \infty$ such that for every $R \ge 2$,
\[c_1 R^{5-4a} \le \Pr{\rho \in \mathcal{C}_R} \le c_2 R^{5-4a} \log R.\]
\end{prop}

\begin{proof}
Since $\rho$ belongs to $\mathcal{C}_R$ if and only if one of its incident edges is an $R$-cut edge, then the lower bound directly follows from Lemma~\ref{lem:racine_cut_edge}. For an upper bound, first notice that the map $\Map_\infty$ is invariant under re-rooting around $\rho$ in the sense that if one replaces the root edge by any other edge incident to $\rho$, and oriented from $\rho$, this new map has the same law as $\Map_\infty$. Therefore the bounds in Lemma~\ref{lem:racine_cut_edge} are valid for all the edges incident to $\rho$. By splitting according to the degree of $\rho$, a union bound yields for every $K > 0$,
\begin{align*}
\Pr{\rho \in \mathcal{C}_R}
&\le \Pr{\deg(\rho) > K \log R} + K \log R \cdot \Pr{\text{the root edge is an } R \text{-cut edge}}
\\
&\le \ex^{-c K \log R} + K \log R \cdot C  R^{5-4a},
\end{align*}
where the second inequality follows from~\eqref{eq:rootdegreevertex} and Lemma~\ref{lem:racine_cut_edge}. We conclude by choosing $K$ large enough so that the last line is smaller than some constant times $R^{5-4a} \log R$.
\end{proof}

It remains to consider Assumption~\ref{H_geom} of Lemma~\ref{lem:conditions_generales_sous_diff}. In this simple setting, the graph induced on the set $\mathcal{C}_R$ simply consists in a discrete one-dimensional chain, and we aim at controlling its length in the ball of radius $R$.
We shall need the following lemma. For $\ell \ge 1$, let $\Map_\infty^{(\ell)}$ be an infinite Boltzmann map with a boundary of length $2 \ell$ and let us denote by $N_\ell$ the number of cut edges which belong to the root face and which separate the origin $\rho$ of the root edge from $\infty$.

\begin{lem}\label{lem:cut_edges_carte_bord}
There exists $\delta, \kappa >0$  such that for every $\ell \ge 1$, in $\Map_\infty^{(\ell)}$, we have
\[\Pr{N_\ell \ge \kappa \ell^{2-a}} \ge \delta.\]
\end{lem}

\begin{proof}
Let us label the edges on the boundary of $\Map_\infty^{(\ell)}$ from $1$ to $2\ell$ in clockwise order, starting from the root edge. Note that when peeling one of these edges, it can be identified with another one only if their label have different parity. 
The cut edges counted by $N_\ell$ are given by those pairs $1 \le i < j \le 2\ell$ with different parity such that the edge $i$ is peeled and gets identified with $j$ during an event $ \mathsf{G}_{\cdot,\cdot}$ and so that the infinite part is separated from the origin. From the exact transition probabilities recalled in Section~\ref{sec:peeling} this happens with probability 
\[\Pr{\mathsf{G}_{*,(j-i-1)/2}} =  \frac{1}{2}\frac{h^\uparrow((j-i-1)/2)}{h^\uparrow(\ell)} \nu\left(\frac{j-i-1}{2}-\ell\right).\]
Summing over all possible pairs and splitting according to the parity of $i$, recalling that $\nu(-k) \sim \mathsf{p}_\q k^{-a}$ and $h^\uparrow(k) \sim c' \sqrt{k}$ for some $c' > 0$ we easily find that 
\[\Es{N_{\ell}} \sim c_{1} \ell ^{2-a},\]
for some $c_{1}>0$ as $\ell \to \infty$. 

Let us next turn to the second moment of $N_\ell$. Now we need to consider pairs of identified pairs of edges; note that the identifications must be planar in the sense that for $1 \le i < j < k < l \le 2\ell$, one cannot identify $i$ with $k$ and $j$ with $l$. Moreover, if we want both these identifications to create cut edges which separate the origin from infinity, then one can only identify $i$ with $l$ and $j$ with $k$; this necessitates that $i$ and $l$ have different parity, and also $j$ and $k$.
In this case, the probability to identify $i$ with $l$ and $j$ with $k$ equals the probability of the event $\mathsf{G}_{*,(l-i-1)/2}$ starting with a half-perimeter $\ell$, times the probability of the event $\mathsf{G}_{*,(k-j-1)/2}$ starting with perimeter $l-i-1$, that is explicitly
\begin{align*}
&\frac{1}{2}\frac{h^\uparrow((l-i-1)/2)}{h^\uparrow(\ell)} \nu\left(\frac{l-i-1}{2}-\ell\right)
\cdot \frac{1}{2}\frac{h^\uparrow((k-j-1)/2)}{h^\uparrow(l-i-1)/2} \nu\left(\frac{k-j-1}{2}-\frac{l-i-1}{2}\right)
\\
&= \frac{1}{4}\frac{h^\uparrow((k-j-1)/2)}{h^\uparrow(\ell)} \nu\left(\frac{l-i-1}{2}-\ell\right) \nu\left(\frac{k-j-1}{2}-\frac{l-i-1}{2}\right).
\end{align*}
With the same reasoning, we obtain that $\E[N_\ell^2] \sim c_2 \ell^{4-2a}$ for some $c_2 > 0$.
Appealing to the Paley--Zygmund inequality, we conclude that for every $\ell$ large enough, we have
\[\Pr{N_\ell \ge \frac{c_1}{2} \ell^{2-a}} \ge \frac{c_1^2}{8c_2},\]
and the proof is complete.
\end{proof}

We may now provide Assumption~\ref{H_geom} of Lemma~\ref{lem:conditions_generales_sous_diff}.

\begin{prop}
\label{prop:cut_edges}
Let $a \in (\frac{3}{2}, 2)$. For every $\varepsilon > 0$, there exists $K \ge 1$ such that for every integer $R$ large enough, with probability at least $1-\varepsilon$, there exist at least $R^{4-2a}$ cut edges within distance $K R$ from the origin in $\Map_{\infty}$ which separate from infinity a portion of the map with volume at least $R^{2a-1}$.
\end{prop}

Note that on the event in the proposition, there are at least $R^{4-2a}$ vertices in $\mathcal{C}_R$ inside the ball of radius $K R$.

\begin{proof}
In $\Map_{\infty}$, let us perform the peeling $( \e_{n})_{n \ge 0}$ with algorithm $\mathcal{A}_{\mathrm{dual}}$ (although the statement of the proposition deals with primal distances) and recall that we denote by $\theta_{k}$ the first time at which no edge is adjacent to a face at dual distance $k$, that is the time it takes to complete $k$ turns for the peeling by layers on the dual map. Let us write $\mathcal{P}_{k}$ for the half-perimeter of the hole of $\e_{\theta_{k}}$. The results of Section~5.2.1 in~\cite{Budd-Curien:Geometry_of_infinite_planar_maps_with_high_degrees} show that there exists $c>0$ so that for any $k \ge 1$, conditionally on the past before $\theta_{k}$, the following scenario happens with probability at least $c>0$:
\begin{itemize}
\item during the next turn, i.e.~between time $\theta_{k}$ and $\theta_{k+1}$, we discover a large face $ \mathrm{f}$ of degree $\ge 4 \mathcal{P}_{k}$,
\item during the second turn, i.e.~between time $\theta_{k+1}$ and $\theta_{k+2}$, two edges of $ \mathrm{f}$ get identified and create a cut edge separating the origin from infinity and so that the remaining hole has half-perimeter at least $\mathcal{P}_{k}$.
\end{itemize}
By the Markov property applied when discovering such cut edge and by Lemma~\ref{lem:cut_edges_carte_bord}, with probability at least $\delta>0$, this face will further create $\kappa (\mathcal{P}_{k})^{2-a}$ additional cut edges during the completion of the turn, i.e. before time $\theta_{k+2}$. Whence for any $ \varepsilon>0$ we can find $M \ge 0$ so that with probability at least $1- \varepsilon$ we discover $\kappa \left(\inf_{k \le i \le k+M} \mathcal{P}_{i}  \right)^{2-a}$ cut edges within time $\theta_{k}$ and $\theta_{k+M}$. Furthermore $\inf_{i \ge k} \mathcal{P}_{i} \ge \varepsilon \mathcal{P}_{k}$ with probability of order $1- \sqrt{ \varepsilon}$ by the proof of Proposition~A.11 in~\cite{Curien:StFlour}. Let us sum up this discussion: For any $ \varepsilon>0$ there exists $M, \delta >0$ so that independently of the past before $\theta_{k}$, there is a probability at least $1- \varepsilon$ that we discover $\delta (\mathcal{P}_{k})^{2-a}$ edges during the next $M$ turns of the algorithm i.e. before $\theta_{k+M}$. Let us wait for the perimeter process to reach values of order $r^{2}$ and volume at least $r^{2a-1}$, which takes time of order $r^{2(a-1)}$ by~\cite[Theorem~3.6]{Budd-Curien:Geometry_of_infinite_planar_maps_with_high_degrees}, up to further adding $M+1$ turns of the peeling by layers, which takes time of order $r^{2(a-1)}$, we have discovered our desired $ r^{4-2a}$ different $r$-cut edges before $K r^{2(a-1)}$ peeling steps (with algorithm $ \mathcal{A}_{ \mathrm{dual}}$) with probability $1-  \varepsilon$ for some large constant $K \ge 0$. 

We now invoke~\cite{Curien-Marzouk:Markovian_explorations_of_random_planar_maps_are_roundish} which shows   that any Markovian exploration is ``roundish'' and grows roughly like metric balls for the primal metric in $\Map_{\infty}$. In particular, by~\cite[Theorem~1]{Curien-Marzouk:Markovian_explorations_of_random_planar_maps_are_roundish}, there exists $A>0$ such that for any $r$ large enough and for any peeling algorithm, $\e_{K r^{2(a-1)}}$ is contained in the primal ball of radius $A r$ with probability at least $1- \varepsilon$. 
The statement of the proposition follows from this remark combined with the conclusion of the preceding paragraph.
\end{proof}

\begin{rem}\label{rem:recurrence}
Using~\eqref{eq:rootdegreevertex} and the work of Gurel-Gurevich \& Nachmias~\cite{Gurel_Gurevich-Nachmias:Recurrence_of_planar_graph_limits} it follows that, in the whole range $\frac{3}{2} < a \le \frac{5}{2}$, the random walk on $\Map_\infty$ is recurrent, see e.g.~\cite{Bjornberg-Stefansson:Recurrence_of_bipartite_planar_maps,Stephenson:Local_convergence_of_large_critical_multi_type_Galton_Watson_trees_and_applications_to_random_maps}. In the range $\frac{3}{2} < a < 2$ this also follows from the preceding proposition since the effective resistance between the root and the boundary $\partial \hBall(\Map_\infty, KR)$ grows at least as $R^{4-2a}$ (up to a $\log R$ factor for the vertex degrees).
\end{rem}

Let us end this section with the proof of Theorem~\ref{thm:sous_diff_primal} in the dense phase, appealing Lemma~\ref{lem:conditions_generales_sous_diff}.

\begin{proof}[Proof of Theorem~\ref{thm:sous_diff_primal} when $a < 2$]
First, the map $\Map_\infty$ has polynomial growth, of order $R^{2a-1}$~\cite[Proposition~2]{Curien-Marzouk:Markovian_explorations_of_random_planar_maps_are_roundish}, whence Condition~\ref{H_degrees} of Lemma~\ref{lem:conditions_generales_sous_diff} is satisfied with any $d > 2a-1$. Next recall that we have defined $\mathcal{C}_r$ as the set of all the extremities of those cut edges which separate from infinity a part of the map with volume at least $r^{2a-1}$.
According to Proposition~\ref{prop:sommet_racine_cut_edge} there exists $C > 0$ such that $\P(X_n \in \mathcal{C}_r) \le C r^{5-4a} \log r$ for all $n \ge 0$ and $r \ge 2$.
Fix $\varepsilon > 0$; according to Proposition~\ref{prop:cut_edges} and the remark just after, there exists $K \ge 1$ such that for every $r$ large enough, with probability at least $1-\varepsilon/2$ there are at least $r^{4-2a}$ vertices in $\mathcal{C}_r$ inside the ball of radius $K r$, and of course each of them must visited before exiting this ball.

Then Lemma~\ref{lem:conditions_generales_sous_diff} applied with $R = \lceil Kr \rceil$ and $\mathcal{G}_R = \mathcal{C}_r$ shows that for every $R$ large enough, with probability at least $1-\varepsilon$, the random walker $X_i$ stays within distance $R$ from the origin for every 
$i \le (C r^{5-4a} \log r)^{-1} (r^{4-2a})^2 \log^{-7/4} R
\le C' R^{3} \log^{-11/4} R$ for some $C' > 0$.
\end{proof}

\section{Subdiffusivity via horocycles in the dilute phase}
\label{sec:dilue}

We presented informally in Section~\ref{sec:heuristic} a strategy which holds for all $a \in (\frac{3}{2}, \frac{5}{2}]$ based on the representation of $\Map_\infty$ ``from infinity''. As alluded there, in order to avoid the precise construction of this object, only available for the UIPQ/T~\cite{Curien-Menard-Miermont:A_view_from_infinity_of_the_uniform_infinite_planar_quadrangulation, Curien-Menard:The_skeleton_of_the_UIPT_seen_from_infinity} we rely on another approximation of $\Map_\infty$ by finite maps which is due to Budd~\cite{Budd:Cours_peeling_Lyon}, which we next present.

\subsection{Boltzmann maps with an edge as target}

In this section we shall consider finite maps with a root face with degree $2\ell$ and another marked face $\mathsf{f}_{1}$ with degree $2$. One can adapt  in a straightforward way the Boltzmann law to this case and define $\Map_1^{(\ell)}$ such a random map with free volume, see~\cite[Section~4.2]{Curien:StFlour} for details. Using the zipping operation (see Figure~3.2 in~\cite{Curien:StFlour}) those maps will also be seen as maps with a distinguished (non-oriented) edge which we will denote by $\mathbf{a}$ and $ \vec{\mathbf{a}}$ after orienting it in a uniformly random fashion amongst the two possibilities.  

We define filled-in peeling processes $(\e_n)_{n \ge 0}$ of such maps (starting from the root face of degree $2\ell$) in the very same way as in Section~\ref{sec:peeling}, see~\cite{Budd:Cours_peeling_Lyon} or~\cite[Chapter~5]{Curien:StFlour} for details. The only difference here is that, at each step, once the peel edge on the boundary of the hole of $\e_n$, is selected, there are now three possibilities:
\begin{itemize} 
\item Either the peel edge is incident to a new face in $\Map_1^{(\ell)}$, different from the distinguished face $\mathsf{f}_{1}$, and then $ \e_{n+1}$ is obtained from $ \e _{n}$ by gluing this face on the peel edge without performing any other identification;
\item Or the peel edge is incident to another face of $ \e_{n}$ in the map $\Map_1^{(\ell)}$, in which case we first perform the identification of the two boundary edges of $\e_n$ and then fill-in the hole which does not contain the face $\mathsf{f}_{1}$;
\item Or the peel edge is incident to the distinguished face $\mathsf{f}_{1}$ in $\Map_1^{(\ell)}$, then we first add this face and then we fill-in the remaining hole and we stop.
\end{itemize}
As in Section~\ref{sec:peeling}, one can write down the probability of each of these events; an important feature is that if $P_n$ denotes the half-perimeter of the hole of $\e_n$ for every $n \ge 0$, then the process $(P_n)_{n \ge 0}$ is a version of the $\nu$-random walk started from $\ell$ and conditioned to first enter $\Z_{\le 0}$ at the point $-1$ and killed there, where the law $\nu$ is defined in~\eqref{eq:nu}. This conditioning is defined as a Doob $h$-transform with the harmonic function $h_1^\downarrow$ where for $p \geq 1$ we have
\begin{equation}\label{eq:fonction_h_cartes_deux_bords}
h^\downarrow_p(k) = 
\frac{k}{k+p} \cdot 2^{-2(k+p)} \binom{2 k}{k} \binom{2p}{p}    \quad\text{for } k \ge0
\qquad\text{and}\qquad
h^\downarrow_p(-p) = 1.
\end{equation}
As in Section~\ref{sec:peeling} one can derive precious information about the lattice $\Map_1^{(\ell)}$ by choosing carefully the peeling algorithm. As an example, the proof of~\cite[Lemma~15.7]{Curien:StFlour} extends easily and shows the analogue of~\eqref{eq:rootdegreevertex} about the degree of the vertex $\rho_{ \vec{ \mathbf{a}}}$ from which $\vec{ \mathbf{a}}$ emanates:
\begin{equation}\label{eq:rootdegreevertex_deux_bords}
\Pr{\deg( \rho_{ \vec{ \mathbf{a}}}) \geq k} \leq \ex^{-ck},
\qquad k \ge 0
\end{equation}
with a constant $c>0$ which does not depend on $\ell$.

We shall observe the map from $ \vec{ \mathbf{a}}$ and  denote by $\vec{\Map}_{1}^{(\ell)}$ the map obtained by forgetting the root edge on the boundary of degree $2 \ell$ and re-rooting the map at $ \vec{ \mathbf{a}}$. The reason why we introduce these random maps is the following result due to Budd~\cite[Theorem~2]{Budd:Cours_peeling_Lyon}. See also~\cite[Theorem~7.1]{Curien:StFlour}.

\begin{prop}[\cite{Budd:Cours_peeling_Lyon}]
\label{prop:convergence_carte_bi_enracinees}
We have $\vec{\Map}_{1}^{(\ell)} \to \Map_{\infty}$ in distribution for the local topology as $\ell \to \infty$.
\end{prop}

We shall use this result in the following context, as depicted in Figure~\ref{fig:horodistances_finies}. Since $ \vec{ \mathbf{a}}$ will play the role of the root edge in $\Map_{1}^{(\ell)}$, we shall use the root face of perimeter $2\ell$ in $\Map_{1}^{(\ell)}$ as playing the role of ``the point at infinity'' in the heuristic discussion in Section~\ref{sec:heuristic}. The conjectural horodistances will simply be replaced by distances to the large boundary $\partial\Map_{1}^{(\ell)}$. 
For our application we shall thus consider the oriented edges $(\vec{E}_{n} : n \ge 0)$ visited by a random walk started from  $\vec{E}_0 = \vec{ \mathbf{a}}$ in $\vec{\Map}_{1}^{(\ell)}$. Since by the zipping operation distinguishing a $2$-gon is the same as distinguishing an edge, it is straightforward that $\vec{\Map}_{1}^{(\ell)}$ is stationary with respect to the random walk i.e. that for every $n \ge 0$, the map obtained from $\vec{\Map}_{1}^{(\ell)}$ by distinguishing $\vec{E}_n$ instead of $\vec{ \mathbf{a}}$ has the same law as $\vec{\Map}_{1}^{(\ell)}$.

\begin{figure}[!ht]\centering
\includegraphics[width=.9\linewidth]{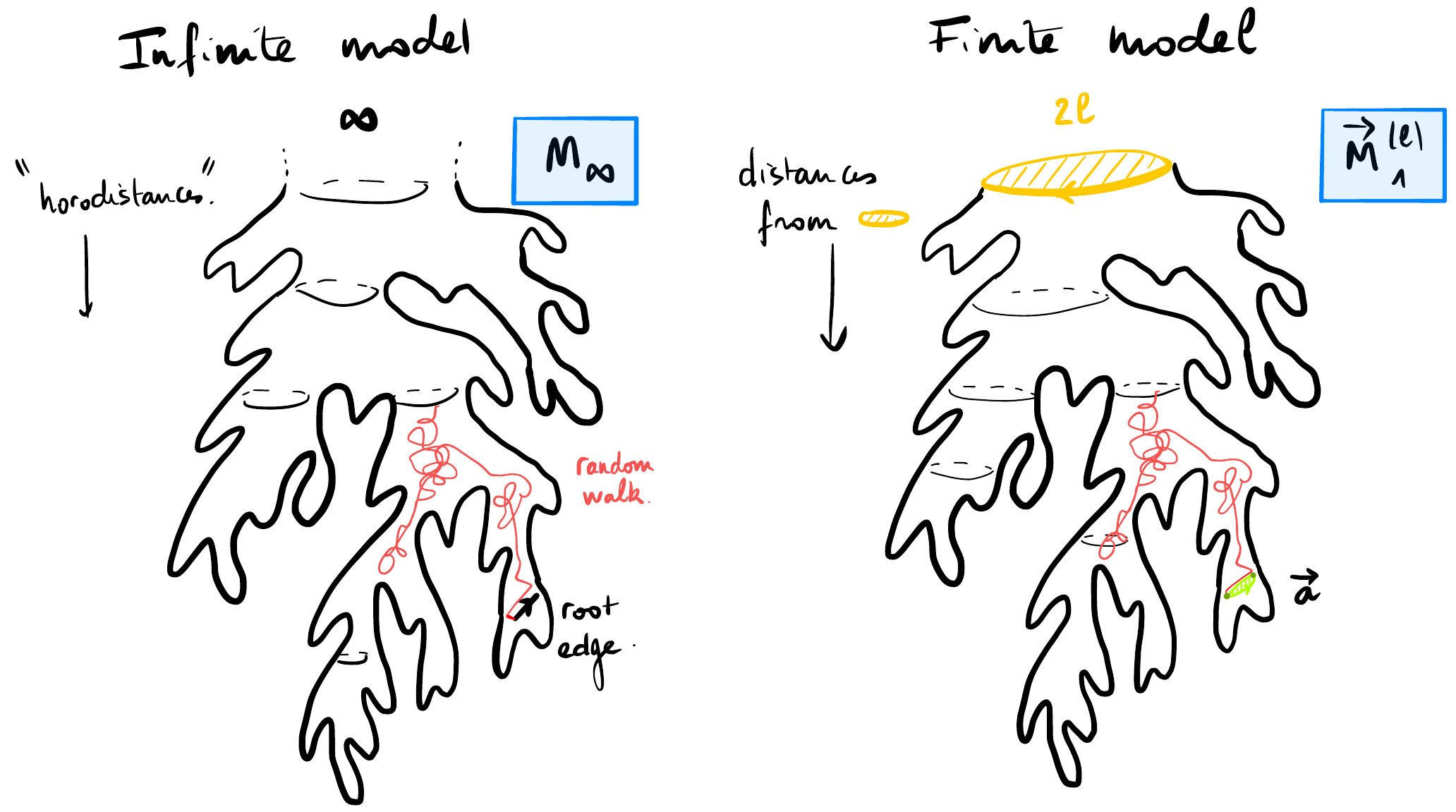}
\caption{Emulating the conjectural horodistances in $\Map_{\infty}$ by distances from a large root face in $\vec{\Map}_{1}^{(\ell)}$.}
\label{fig:horodistances_finies}
\end{figure}

\subsection{Finding a good set}

We will now define a stationary set of ``good edges'' in $\Map_1^{(\ell)}$ by using an exploration with algorithm $ \mathcal{A}_{ \mathrm{metric}}$ recalled in Section~\ref{sec:peeling_layer}. 
We start with the construction in the deterministic setting.\medskip

Fix any finite map $\map^{(\ell)}$ with a boundary face of degree $2 \ell$. We shall measure the distance in $\map^{(\ell)}$ to the boundary $ \partial \map^{(\ell)}$ of degree $2 \ell$. Adapting the algorithm $ \mathcal{A}_{ \mathrm{metric}}$ of Section~\ref{sec:peeling_layer} we shall always peel an edge whose right hand point minimises the distance to $\partial \map^{(\ell)}$. Contrary to the previous cases, we shall not considered the filled-in version of this exploration, and continue our process in each hole thus created: inside each of these holes we peel an edge whose right-hand point minimises the distances (amongst all vertices of that holes) to $\partial \map^{(\ell)}$. We shall freeze the exploration inside a hole as soon as the remaining volume (number of edges) of the map which should fill it in drops below $R^{2a-1}$. Notice that this exploration is not ``Markovian'' since it uses the knowledge of the undiscovered part, but we shall only use it to define our set of good edges. When the exploration is finished, we get a submap $ \mathfrak{e} \subset \map^{(\ell)}$ with holes, and each of these holes hides a map of volume smaller than $ R^{2a-1}$. The set of all edges explored during this process is the set of $R$-good edges.

As in the proof of Lemma~\ref{lem:conditions_generales_sous_diff}, let us consider in $\map^{(\ell)}$  the graph $ \mathcal{H}_{R}$ (for horodistances) spanned by the vertices incident to $R$-good edges and where two vertices are linked by an edge if there exists a path in $\map^{(\ell)}$ going from one to the other without visiting any other vertex of $\mathcal{H}_R$. For $x \in \mathrm{Vertices}(\map^{(\ell)})$ let us also write 
\begin{equation}\label{eq:horodistances_fini}
H(x) = \dgr(x, \partial \map^{(\ell)}),
\end{equation}
the distance from $x$ to the boundary. See Figure~\ref{fig:rgood} Right.

\begin{figure}[!ht]\centering
\includegraphics[width=\linewidth]{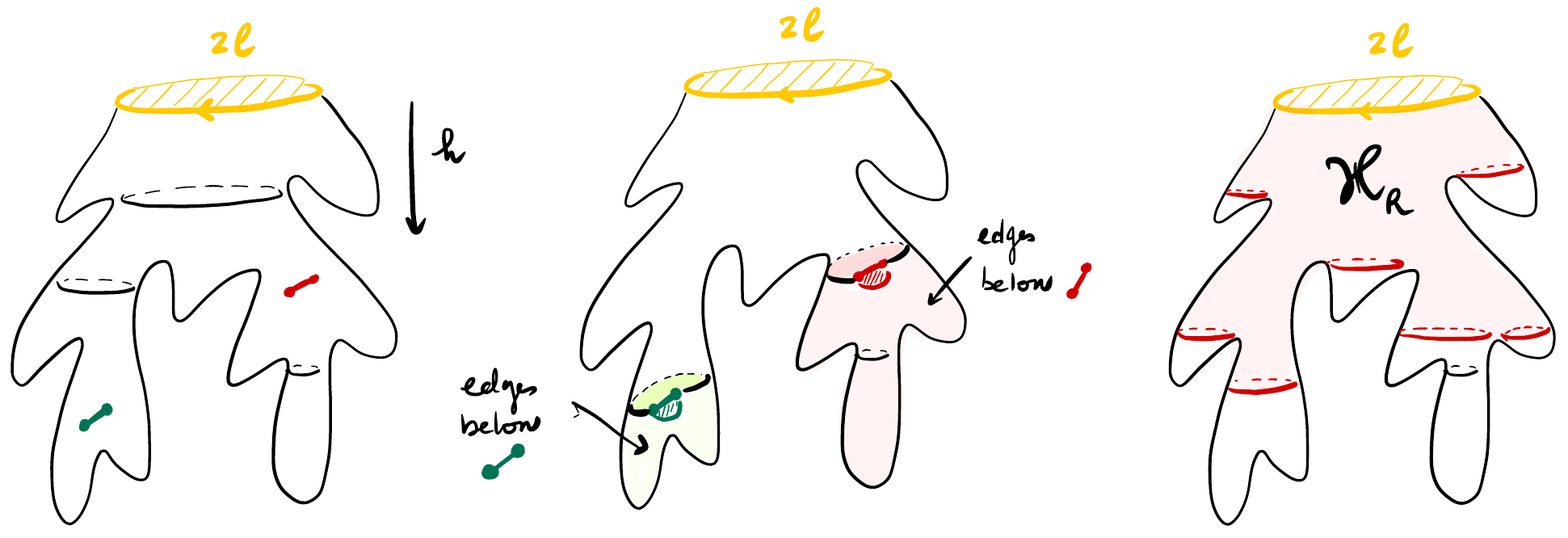}
\caption{Illustration of the construction of $R$-good edges in a map with a boundary as those edges which, when discovered using the filled-in peeling algorithm $ \mathcal{A}_{ \mathrm{metric}}$, possess at least $R^{2a-1}$ edges below them, i.e.~in the remaining hole to be filled-in.}
\label{fig:rgood}
\end{figure}

\begin{prop}\label{prop:distance_horohull}
If $x,y$ are two adjacent vertices in $ \mathcal{H}_{R}$ then $ |H(x)-H(y)| \leq 1$. In words, the graph distances in $ \mathcal{H}_{R}$ are larger than the differences of the distances to the boundary of the map.
\end{prop}

This shows that $\mathcal{H}_{R}$ does not create ``shortcuts'' in the sense that if a path in the map goes from a vertex $x \in \mathcal{H}_{R}$ to another vertex $y \in \mathcal{H}_{R}$ such that $H(y) = H(x) - k$, then the path flashed on $\mathcal{H}_{R}$ goes from $x$ to $y$ in at least $k$ steps. 
In view of applying Lemma~\ref{lem:conditions_generales_sous_diff} with $\mathcal{G}_R = \mathcal{H}_{cR}$ for some $c > 0$, this will provide Assumption~\ref{H_geom} with $\beta_R$ of order $R$.

\begin{proof}
Let us examine the situation after the branching peeling exploration with algorithm $ \mathcal{A}_{ \mathrm{metric}}$ frozen when the volume of the map of a given hole drops below $R^{2a-1}$. The submap $ \mathfrak{e}$ obtained may have several holes, which are simple faces which cannot share any edge but may share vertices with $ \partial \map^{(\ell)}$. Label the vertices of $ \mathfrak{e}$ with respect to their graph distance \emph{within $ \mathfrak{e}$} to the boundary $ \partial \map^{(\ell)}$. The key is to notice that by the properties of algorithm $ \mathcal{A}_{ \mathrm{metric}}$, each hole of $ \mathfrak{e}$ has the following property:
For each hole $h$ of $  \mathfrak{e}$ there exists an integer $ a \geq 0$ such that the vertices which are adjacent to another vertex of $ \mathfrak{e} \setminus h$ carry either label $a$ or $a+1$. See Figure~\ref{fig:hole}.  The vertices of label $a$ or $a+1$ inside a given hole are called \emph{exit vertices} in the following lines.

\begin{figure}[!ht]\centering
\includegraphics[width=10cm]{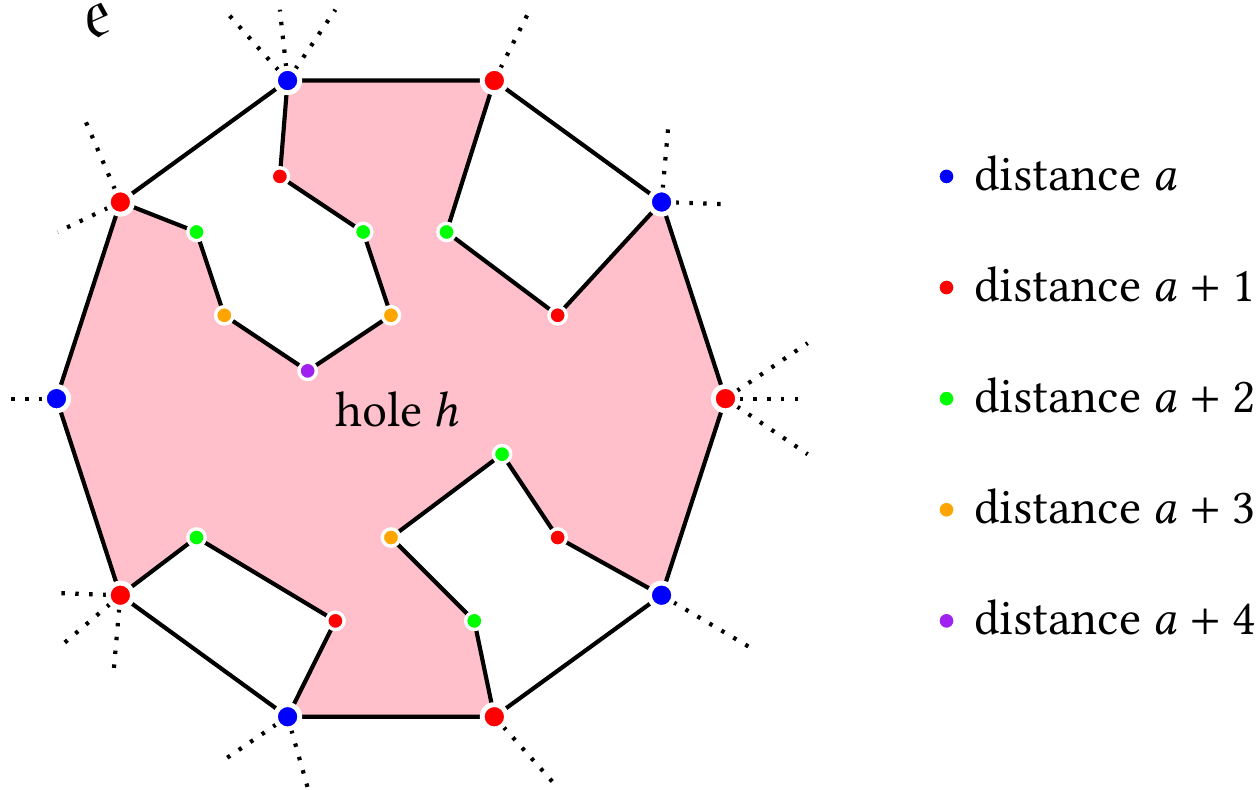}
\caption{Typical status of the distances to the boundary of the map along a given hole. Notice that only the vertices with label $a$ or $a+1$ --the exit vertices-- can be linked to a different part of $ \mathfrak{e}$.}
\label{fig:hole}
\end{figure}

This property is easy to prove by induction using the definition of algorithm $ \mathcal{A}_{ \mathrm{metric}}$ and the peeling transitions. Also if $x \in \mathfrak{e}$ is a inner vertex (i.e. not on a hole) or a vertex of a hole with label $a$ or $a+1$ as above, then the graph distance from $x$ to $ \partial \map^{(\ell)}$, coincide inside $ \mathfrak{e}$ and within the larger map $\map^{(\ell)}$. Let us now consider a walk inside $\map^{(\ell)}$ and let us flash it on $ \mathcal{H}_R$. It should be clear from the above that although the transitions for the flashed walk may be arbitrary inside a given hole, the only way to ``escape'' from a hole and walk inside $ \mathcal{H}_R$ is to go through an exit vertex. Whence, by the above property, between a time where the walk enters a hole (or a inner vertex of $ \mathfrak{e}$) and the time it exits it,  its label (distance to the boundary) cannot vary by more than $1$ in absolute value. This prove our proposition.
\end{proof}

\subsection{\texorpdfstring{Stationarity and density of $ \mathcal{H}_{R}$}{Stationarity and density of H\_R}}

We now study the properties of the graph $ \mathcal{H}_{R}$ in the case when $ \map^{(\ell)}$ is the random map $ \vec{\Map}_{1}^{(\ell)}$. First of all, recall that $ \vec{\Map}^{(\ell)}_{1}$ is stationary, i.e. invariant under the simple random walk started from $X_0 = \rho_{\vec{ \mathbf{a}}}$ the origin of the oriented edge $\vec{ \mathbf{a}}$ of $\Map_{1}^{(\ell)}$ so the probability that $X_n$ belongs to $\mathcal{H}_{R}$ does not depend on $n$. 
Let us study the density of $ \mathcal{H}_{R}$ (Assumption~\ref{H_densite} of Lemma~\ref{lem:conditions_generales_sous_diff}) using the filled-in peeling process under $ \Map_{1}^{(\ell)}$:

\begin{prop}\label{prop:proba_arete_good_cas_dilue}
There exist two constants $c, C > 0$ such that, uniformly for $\ell \ge 1$ and $R \ge 1$, inside $ \Map_{1}^{(\ell)}$ we have
\[\Pr{\mathbf{a} \text{ is }R \text{-good}} \le c R^{-1}
\qquad\text{and}\qquad
\Pr{\rho_{\vec{ \mathbf{a}}} \in \mathcal{H}_R} \le C R^{-1} \log R.\]
\end{prop}

\begin{proof}[Proof of Proposition~\ref{prop:proba_arete_good_cas_dilue}]
The second claim follows from the first one as in Proposition~\ref{prop:sommet_racine_cut_edge}, appealing to~\eqref{eq:rootdegreevertex_deux_bords} instead of~\eqref{eq:rootdegreevertex}.
We thus focus on the edge $\mathbf{a}$. To see whether $ \mathbf{a}$ is $R$-good we shall open it and explore the map $\Map_{1}^{(\ell)}$ from the boundary towards the distinguished $2$-gon using a filled-in version $(\e_n)_{n \ge 0}$ of the peeling exploration described in the last section.  At the last exploration step at time $\tau$, we must reveal a $2$-gon corresponding to the edge $ \mathbf{a}$. By definition $ \mathbf{a}$ is $R$-good if the remaining hole is filled-in with a map with at least $R^{2a-1}$ edges (note that this does not depend on the orientation of $ \vec{ \mathbf{a}}$).

This implies that at every preceding time $n < \tau$, the unrevealed map which fits in the hole of $\e_n$ must have volume at least $R^{2a-1}$. By the strong Markov property, if $n$ is a stopping time and the perimeter of the hole of $\e_n$ is, say $2p \ge 2$, then the map that fills-in the hole is independent of the exploration and has the law $\Map_1^{(p)}$. According to~\cite[Theorem~3.12]{Curien:StFlour} the volume of such a map has the following law:
\[\P(|\Map_{1}^{(p)}| = n) = \Pr{\zeta_{-p-1} = n},
\qquad\text{for every } n \ge 1,\]
where $\zeta_{-p-1}$ is the hitting time of $-p-1$ of a random walk with step distribution $\mu(\cdot+1)$ on $\Z_{\ge-1}$ where $\mu$ is defined just before~\cite[Theorem~3.12]{Curien:StFlour}. 
In our case, $\mu$ is centred and belongs to the strict domain of attraction of a stable law with index $a-1/2$ (see Proposition~5.9 in~\cite{Curien:StFlour}), then an application of the cyclic lemma and the local limit theorem shows that $\zeta = \zeta_{-1}$ belongs to the strict domain of attraction of a stable law with index $(a-1/2)^{-1} < 1$ in the sense that there exists a constant $c_a > 0$ such that $\P(\zeta \ge x) \sim c_a x^{- 1/(a-1/2)}$ as $x \to \infty$. Then a \emph{one big jump principle} states that there exists a constant $C > 0$ such that for all $n, p \ge 1$ we have
\[\Pr{\zeta_{-p} \ge n} \le C \cdot p \cdot n^{-1/(a-1/2)},\]
see e.g.~\cite[Theorem~2.2.1]{Borovkov-Borovkov-Borovkova} with $r=1$.

Let us denote by $P^{(\ell)}$ the half-perimeter process of the hole during this exploration (hence $P^{(\ell)}$ starts from $\ell$ and evolves as an $h^{\downarrow}_{1}$-transform of the $\nu$-random walk). By stopping the peeling at the first time this half-perimeter enters an interval of the form $[2^k, 2^{k+1}]$, we obtain that
\[\P( \mathbf{a} \text{ is }R \text{-good})
\le \P\left(\min_{k < \tau} P^{(\ell)}_k \ge R^2\right)
+ \sum_{k < \log_2(R^2)} \P\left(\min_{k < \tau} P^{(\ell)}_k \ge 2^k\right) \sup_{j \in [2^k, 2^{k+1}]}  \P(|\Map^{(j)}_{1}| \ge R^{2a-1}).\]
The tail probability of $\min_{k < \tau} P^{(\ell)}_k$ is bounded in Lemma~\ref{lem:borne_min_marche_cond_mourir_0} below and that of the volume of $|\Map^{(j)}_{1}|$ has just been discussed; we obtain that for some $C > 0$,
\begin{align*}
\P(\mathbf{a} \text{ is }R \text{-good})
&\le \frac{2}{R} + \sum_{k < \log_2(R^2)} \frac{2}{2^{k/2}} \sup_{\ell \in [2^k, 2^{k+1}]} C \cdot \ell \cdot (R^{2a-1})^{-1/(a-1/2)}
\\
&\le \frac{2}{R} + 4C \sum_{k < \log_2(R^2)} 2^{k/2} R^{-2},
\end{align*}
which indeed is bounded by some constant times $R^{-1}$.
\end{proof}

In the course of the proof, we used the following lemma.

\begin{lem}\label{lem:borne_min_marche_cond_mourir_0}
For every integers $\ell \ge m \ge 1$, we have that
\[\Pr{\min_{k < \tau} P^{(\ell)}_k \ge m} \le \frac{2}{\sqrt{m}}.\]
\end{lem}

\begin{proof} Let $(S_k)_{k \ge 0}$ be a random walk with i.i.d.~increments of law $\nu$ which, under $\P^{(m_{0})}$ starts from $m_{0}$. By definition of the $h$-transform, we have that
\begin{align*}
\Pr{\min_{k < \tau} P^{(\ell)}_k \ge m}
&= \frac{h_1^\downarrow(-1)}{h_1^\downarrow(\ell)} \cdot \P^{(\ell)}\left(\min_{k < \tau} S_k \ge m \text{ and } \tau_{-1} = \tau < \infty\right)
\\
&= \frac{1}{h_1^\downarrow(\ell)} \cdot \P^{(\ell-m+1)}\left(\min_{k < \tau} S_k \ge 1 \text{ and } \tau_{-m} = \tau < \infty\right)
\\
&= \frac{1}{h_1^\downarrow(\ell)} \cdot h_{m}^\downarrow(\ell-m+1),
\end{align*}
where the last line follows from~\cite[Proposition~5.3]{Curien:StFlour} and we recall the functions $h_p^\downarrow$ from~\eqref{eq:fonction_h_cartes_deux_bords}; recall also $h^\uparrow$ from~\eqref{eq:h_up}. We may re-write this as
\[\Pr{\min_{k < \tau} P^{(\ell)}_k \ge m}
= \frac{h^\uparrow(\ell-m+1)}{h^\uparrow(\ell)} \cdot \frac{h^\uparrow(m)}{h^\uparrow(1)} \cdot \frac{1}{m}.\]
Using that $h^\uparrow$ is increasing, the first ratio is bounded by $1$, the claim then follows from the easy bounds $\sqrt{k} \le h^\uparrow(k) \le 2\sqrt{k}$ for every $k \ge 1$.
\end{proof}

\subsection{Proof of Theorem~\ref{thm:sous_diff_primal} in the dilute case}

Let us end this paper with the proof of Theorem~\ref{thm:sous_diff_primal} in all regimes $3/2 < a \le 5/2$; we illustrate the argument in Figure~\ref{fig:proof_dilute}.

\begin{proof}[Proof of Theorem~\ref{thm:sous_diff_primal}] Let us write $ \mathcal{B}_R$ for the ball of radius $R$ in $ \Map_\infty$ and by $\tilde{ \mathcal{B}}^{(\ell)}_R$ the ball of radius $R$ around $\rho_{ \vec{ \mathbf{a}}}$ in $\vec{\Map}_1^{(\ell)}$. Fix $\varepsilon > 0$. We aim at showing that, on $\Map_\infty$, with probability at least $1 - \varepsilon$, for some large $K \geq 1$, when $R$ is large enough, after $R^3 \log^{-11/4} R$ steps, the random walk has not escaped from $ \mathcal{B}_{KR}$ with probability at least $1- \varepsilon$. The constant $K$ will be chosen below but notice already that for every $K,R \geq 1$, by Proposition~\ref{prop:convergence_carte_bi_enracinees}, we can chose $\ell \equiv \ell( K\cdot R)$ and couple the realisations of $\Map_\infty$ and $\Map^{(\ell)}_1$ in such a way that $ \mathcal{B}_{KR}$ coincides with $ \tilde{ \mathcal{B}}_{KR}^{(\ell)}$ with probability at least $ 1-\varepsilon/10$. Of course on this event, we can further suppose that the random walkers in both graphs coincide up to their first exit time of those balls. This coupling enables us to transfert properties from $\Map_\infty$ to $\Map^{(\ell)}_1$ and to use the random stationary set $ \mathcal{H}_R$ in the latter. 

In particular, by~\cite[Proposition~2]{Curien-Marzouk:Markovian_explorations_of_random_planar_maps_are_roundish} we know that $\Map_\infty$ has polynomial growth of order $R^{2a-1}$ and so we can chose $K \geq C\geq 1$ so that for all $R$ sufficiently large, the volume (in terms of number of edges) of the balls satisfies $|\mathcal{B}_{CR}| \geq R^{2a-1}$ and $|\mathcal{B}_{KR}| \leq R^{10}$ with probability $ 1-\varepsilon/10$. Similarly, up to further increasing $K$, the results of~\cite{Curien-Marzouk:Markovian_explorations_of_random_planar_maps_are_roundish} show that at the first exit time of $ \mathcal{B}_{KR}$, the trace of the random walk has already separated $ \mathcal{B}_{2CR}$ from from $\infty$ with probability at least $1- \varepsilon/10$. By the above coupling, we can chose $ \ell$ large enough so that the above properties holds in $ \vec{\Map}_1^{(\ell)}$ with probability at least $1- 3 \varepsilon/10$. 

Let us now work in $ \vec{\Map}_1^{(\ell)}$. When the above conditions are satisfied, we know that when exiting $ \tilde{ \mathcal{B}}_{KC}^{(\ell)}$ the random walk has entirely surrounded $ \tilde{ \mathcal{B}}_{2CR}^{(\ell)}$, and so it must have visited a vertex $x$ with $H(x) \leq H(\rho_{ \vec{a}}) - 2CR$. Since $|\mathcal{B}_{CR}| \geq R^{2a-1}$, by Proposition~\ref{prop:distance_horohull} this implies that the random walker must have travelled for a distance at least $CR$  through  $ \mathcal{H}_R$ before exiting $ \tilde{ \mathcal{B}}_{KC}^{(\ell)}$.
Using the above volume estimates together with Proposition~\ref{prop:proba_arete_good_cas_dilue}, Lemma~\ref{lem:conditions_generales_sous_diff} shows that with probability at least $1-\varepsilon/10$, when $R$ is large, this has necessitated at least $R^3 \log^{-11/4} R$ steps of the random walk.  Using the coupling between $ \mathcal{B}_{KR}$ and $ \tilde{ \mathcal{B}}_{KR}^{(\ell)}$ we deduce that with probability at least $1 - \varepsilon$, after $R^3 \log^{-11/4} R$ steps, the random walk has not yet escaped from $ \mathcal{B}_{KR}$.
\end{proof}

\begin{figure}[!ht]\centering
\includegraphics[width=12cm]{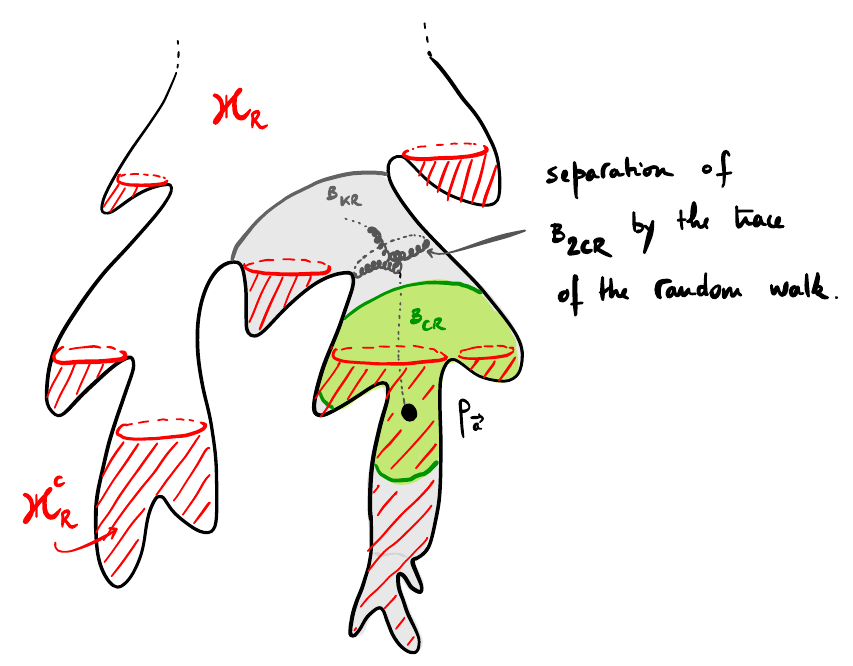}
\caption{Illustration of the proof of Theorem~\ref{thm:sous_diff_primal}. When exiting the ball of radius $KR$, with high probabiliy, the random walk trace must have separated the ball of radius $2CR$ from infinity. If $ | \mathcal{B}_{CR}| \geq R^{2a-1}$ this means that the walker must have travelled for a distance at least $CR$ inside $ \mathcal{H}_R$.}
\label{fig:proof_dilute}
\end{figure}


{
\linespread{1}\selectfont

}

\end{document}